\documentclass{article}

\usepackage[letterpaper,top=2cm,bottom=2cm,left=1.5cm,right=1.5cm,marginparwidth=0.5cm]{geometry}

\usepackage{graphicx} 
\usepackage{amsmath}
\usepackage{graphicx}
\usepackage{amsfonts}
\usepackage[colorlinks=true, allcolors=blue]{hyperref}
\usepackage{amsthm}
\usepackage{mathtools}
\usepackage{bbold}
\usepackage[backend=biber, giveninits=true, url=false,doi=false]{biblatex}
\usepackage{nicematrix}

\newtheorem{theorem}{Theorem}[section]
\newtheorem{Lemma}[theorem]{Lemma}
\newtheorem{Proposition}[theorem]{Proposition}
\newtheorem{Corollary}[theorem]{Corollary}

\theoremstyle{definition}
\newtheorem{Definition}[theorem]{Definition}
\newtheorem{Example}[theorem]{Example}

\theoremstyle{remark}
\newtheorem{Remark}[theorem]{Remark}

\DeclareMathOperator{\Jac}{Jac}

\DeclareMathOperator{\vol}{vol}
\DeclareMathOperator{\Ric}{Ric}
\DeclareMathOperator{\Fluct}{Fluct}

\makeatletter
\renewcommand{\@makefnmark}{}
\makeatother

\title{Free energy of the Coulomb gas in the determinantal case on Riemann surfaces}
\author{Lucas Bourgoin }
\date{\today}

\begin{document}
\maketitle
\vspace{-2.65em}
\begin{abstract}
    We derive the asymptotic expansion of the partition function of a Coulomb gas system in the determinantal case on compact Riemann surfaces of any genus $g$. Our main tool is the bosonization formula relating the analytic torsion and geometric quantities including the Green functions appearing in the definition of this partition function. As a result, we prove the geometric version of the Zabrodin-Wiegmann conjecture in the determinantal case.
\end{abstract}
\tableofcontents
\section{Introduction}
The Coulomb gas describes a system of $N$ charged particles interacting through a Coulomb interaction. Many properties of the Coulomb gas have been studied by several authors, for a recent account see \cite{serfaty2024lecturescoulombrieszgases,Ginibre} and references therein. In this article we are interested in this system on a compact Riemann surface $M$ of genus $g$. More precisely, we are interested in finding an asymptotic expansion of the logarithm of the partition function $\mathcal{Z}$ when the number $N$ of particles tends to infinity. We define the Coulomb gas partition function on the Riemann surface $M$ as follows: 
\begin{equation}\label{Eq:Z_Beta}
\mathcal{Z}_{\beta}=\frac{1}{N!}\int_{M^{N}}\mu^{N}\exp\left(2\beta\sum_{1\leq i<j\leq N} G(z_{i},z_{j})+\beta(N+g-1)\sum_{i=1}^{N} V(z_{i})\right)
\end{equation}where $G$ is the Green function of the scalar Laplacian: $$\Delta G(x,y)=-\frac{2\pi\delta(x-y)}{\sqrt{\det \rho}}+\frac{2\pi}{\vol(M,\rho)},  $$ with $\Delta$ the positive scalar Laplacian in some metric $\rho.$ $\beta >0$ is a real parameter physically corresponding to the inverse of the temperature, in the following, we will only consider $\beta=1.$ $V \in \mathcal{C}^{\infty}(M;\mathbb{R})$ is what we will call a quasi-subharmonic potential: $$\Delta V>-4\pi  ~~\text{ (see Equation \eqref{ConditionV})}$$ possibly depending on $N$ and $\mu$ is a volume form on $M$. For simplicity in our computation, we will consider $\mu$ to be the volume form associated with the Arakelov metric, see Section \ref{SectionArakelovMetric} for definition. The constant $N+g-1=k$ in front of $V$ in the definition of $\mathcal{Z}_{\beta}$ (Equation \eqref{Eq:Z_Beta}), corresponding to the degree of a Hermitian line bundle, proves to be a particularly convenient choice for our computation (see Equation \eqref{eq:HermitianPotential}). Our computations would still hold to be true with the standard constant $N$ before the Laplacian but it would make the computations and the formulae more cumbersome. We also note our choice to write $G$ for the Green function, for the notational consistency with Coulomb gas  literature, as opposed to the Arakelov geometry and bosonization formula literature, where the Green function is denoted by $\ln G$ (see \cite{ArakelovSJ_1974,WentworthArakelov,WentworthGluing} for instance). In statistical physics, the free energy is a particularly important quantity to describe systems and knowing the behaviour of this quantity for a large number of particles is of special interest \cite{Jancovici1994CoulombSS}. 

We will be looking for an expansion for $N \to \infty$ of this free energy $\ln \mathcal{Z}_{\beta}$ when $\beta=1$, the so-called determinantal case. We will find an asymptotic expansion of the form: 
$$\ln \mathcal{Z}_{\beta=1}=b_{2}N^{2}+a_{1} N\ln N + b_{1} N + a_{0} \ln N +b_{0} +o(1)$$ where the different terms are explicitly computed. The absence of $\sqrt{N}$ term is consistent with predictions in the determinantal case (see \cite[§9.3]{serfaty2024lecturescoulombrieszgases}). Our main result is that the coefficient $b_{0}$ turns out to contain the logarithm of the regularized zeta determinant of the scalar Laplacian. When $M=\mathbb{C}$, this term was conjectured to appear by Zabrodin and Wiegmann \cite{Zabrodin_2006}. More precisely, they considered a system on the plane of $N$ interacting Coulomb charges subject to a potential $W$. In this case, the partition function was $$Z_{N}=\int_{\mathbb{C}^{N}}|\Delta_{N}(z_{i})|^{2\beta}\prod_{j=1}^{N}e^{NW(z_{j})}d^{2}z_{j}$$ where $\Delta_{N}$ is the Vandermonde determinant and $\beta=1/T$ is the inverse temperature. They made a prediction for the free energy expansion: $$\ln Z_{N}=N^{2}F_{0}+N\ln(N)\tilde{F}_{1/2}+ N  F_{1/2}+\ln(N)\tilde{F}_{1}+F_{1}+O(N^{-1})$$ with explicit predicted formulae for the different terms of the expansion (see also \cite{Jancovici1994CoulombSS,ExactAsymptCoulombCanForresterTellezWiegmann2015,serfaty2024lecturescoulombrieszgases} for other presentations of this conjecture). In the case $M=\mathbb{C}$, the first two terms of the expansion have been computed for general $\beta$ and relations with the Gaussian free field have been shown \cite{LebléSerfatyLDP,AmeurHedenmalMakarov2011Fluctuations}.

Apart from setting $\beta=1$, another simplification in our computation is that we only consider quasi-subharmonic potentials (see Equation \eqref{ConditionV}), as explained above. This will correspond to the case where the support of the droplet is the whole Riemann surface instead of a compact domain of $\mathbb{C}$ as in \cite{Zabrodin_2006}.

 The case of $M$ being a compact Riemann surface is sometimes called the geometric Zabrodin-Wiegmann conjecture. It was introduced in \cite{Klevtsov2014RandomNormalMatrices} and was further studied in \cite{KlevtsovQuillen}, also in the determinantal case. We also note that the case of integer $\beta$ corresponds to the Laughlin states in the quantum Hall effect \cite{KlevtsovGenus,LiebRougerieLocalIncompressibLaughlin} and the version of the geometric Zabrodin-Wiegmann conjecture in this case can be found in \cite{KlevtsovFerrari2014, CanLaskinWiegmann2015,Shen:GeomZabrodin2025}. In \cite{Klevtsov2014RandomNormalMatrices,KlevtsovQuillen}, using the results of \cite{BismutGilletSouleQuillenMetric} the authors studied the variation of the partition function when changing the metrics on the Riemann surface and proved that the Liouville functional appeared in the variation of the constant term of the expansion. This result was indicating the presence of the determinant of the Laplacian in this term because the Liouville functional is involved in the anomaly, namely the variation of the determinant of the Laplacian when the metric changes, of the determinant of the scalar Laplacian (see for instance \cite{OSGOOD1988148,POLYAKOV1981207}, and Proposition \ref{VarDetSca}). In our main result, we will see the explicit expression of the constant term with the determinant of the scalar Laplacian appearing thus confirming the previous predictions. The fact that the determinant of the scalar Laplacian appears in the constant term points toward a relation with the Gaussian free field. The question investigated in this article is related to problems in random matrix theory, particularly for the Ginibre ensemble \cite{ForresterLogGasesAndRandomMatrices,Ginibre}, which has led to many works in many different settings concerning the potential $V$ (see for instance \cite{ameur2025twodimensionalcoulombgasfluctuations,Byun2024FreeEnergy,AllardForresterLahiryShenPartitionFunction,ByunAnomalousFreeEnergy,BauerschmidtBourgadeNikulaYayTwoDimCoulombQuasiFreeCLT,BBNYLocalDensity2D1OCP,RougerieFreeEnergyHoles}) and some analogous results have been found for the case of the sphere \cite{Ginibre,ByunPartition2023,ByunFreeEnergy2025} and of the torus \cite{Forrester_2006}, see also \cite{BorotGuionnetAsymptoticExpansionBetaMultiCut,BorotGuionnetOneCut} for the 1D-case.
 
  We also need to mention that an analogous result was recently obtained for the  Integer Quantum Hall state using a different method \cite{Shen:GeomZabrodin2025}. Similar questions were also studied on manifolds of higher dimension with a Hermitian line bundle (see \cite{Berman2014CMP327DetPPLDP+Boso} and the recent article \cite{Eum2025PartitionFunctionDPP}). In \cite{Berman2014CMP327DetPPLDP+Boso, Eum2025PartitionFunctionDPP, Shen:GeomZabrodin2025}, the authors studied determinantal point processes associated with a positive Hermitian line bundle, which is particularly related to the study of Quantum Hall effect. In our computation, we consider a Coulomb gas on a compact surface and we will obtain an asymptotic expansion of the partition function of the Coulomb gas which will not depend on a line bundle (see Theorem \ref{TheoremFinalN}). More precisely, \cite{Eum2025PartitionFunctionDPP} computed the variation of the expansion of a determinantal point process on Kähler manifolds with positive line bundles, which was computed in \cite{KlevtsovQuillen} for the Riemann surfaces case. In the following we will carry out our computations for a specific choice of metrics, namely the Green function will be associated with the canonical metric $\rho_{can}$ and the integration will be made with respect to the Arakelov metric $\rho_{Ar}$, and the partition function will thus be denoted by $\mathcal{Z}_{\rho_{Ar},\rho_{can},N}$ (see Definition \ref{DefFoncPart}). This choice is particularly convenient for the bosonization formula (see below and Section \ref{SectionBoso}). Moreover, since modifying the metric is equivalent to changing the potential $V$ it is enough to do the computations for one specific choice of metric.
  
\subsection{Main result}

Our main result is given in Theorem \ref{ThéorèmeFinalSansTheta} (see also Theorems \ref{TheoremFinalK} and \ref{TheoremFinalN}) in a precise version. We present here a simplified version of our result.

\begin{theorem}
With $V$ a quasi-subharmonic potential, the partition function satisfies:
$$\ln \mathcal{Z}_{\rho_{Ar},\rho_{can},N}^{}(V)=B_{2}N^{2}+A_{1}N\ln N+B_{1} N+A_{0}\ln N+B_{0}+O(N^{-1}\ln N)
$$
when $N  \to +\infty$, where the constants $A_{i}$ and $B_{i}$ are given explicitly in terms of the genus of $M$ and natural geometric functionals in Theorem \ref{ThéorèmeFinalSansTheta}.
\end{theorem}
The indices $\rho_{Ar},\rho_{can}$ indicate a specific choice of metrics for the computation (see Definition \ref{DefFoncPart}).
To perform our computation, we will introduce a positive line bundle $L$ of degreee $k=N+g-1$ which will not appear in the final expansion and we will use the bosonization formula \cite{Fay1992KernelFA,WentworthGluing,Alvarez-Gaume:1987BosonizationHigherGenus,DHoker:1988pdl,KlevtsovGenus} relating the analytic torsion introduced by Ray-Singer in \cite{RAY1971145,RaySingerAnalyticTorsion1973} (which is in our case the zeta regularised determinant of the magnetic Laplacian $\square_{L,\rho_{Ar},h}$ on some positive Hermitian line bundle $(L,h)$ which will be presented in Section \ref{SectionDeterminantOftheLaplacians}) and different quantities for the Arakelov metric $\rho_{Ar}$ \cite{ArakelovSJ_1974,FaltingsCalculusArithmeticSurface1984}, those objects will be introduced in Section \ref{SectionDeterminantOftheLaplacians} and Section \ref{SectionArakelovMetric}. More precisely, the bosonization formula is given by (see Proposition \ref{Pr:FormuleBoso}): $$\frac{\det \square_{L,\rho_{Ar},h}}{\det \langle \omega_{i}, \omega_{j} \rangle_{\rho_{Ar},h}}= B_{g,k} \frac{\exp\left(2\sum_{1\leq i <j \leq N} G_{Ar}(p_{i},p_{j}) \right)}{\| \det \omega_{i}(p_{j})\|_{h}^{2}}\bigg(\frac{\det\Delta_{\rho_{Ar}}}{\vol(M,\rho_{Ar}) \det \mathfrak{Im}(\tau)}\bigg)^{-1/2}\|\theta\|^{2}\big([L]-\sum_{i=1}^{N}p_{i}-\mathcal{D},\tau\big)$$  with $G_{Ar}$ the Arakelov-Green function (see Section \ref{SectionArakelovMetric}), $(\omega_{i})_{1\leq i\leq N}$ a basis of $H^{0}(M,L)$ (which is of dimension $N=k+1-g$ from Riemann-Roch theorem) and $\|\det\omega_{i}(p_{j})\|_{h}^{2}=|\det\omega_{i}(p_{j})|^{2}h(p_{j}),$ and $p_{j}$ are distinct points on the Riemann surface $M$. The constants $B_{g,k}$, depending only on the genus $g$ of the surface $M$ and the degree $k=N+g-1$ of the line bundle $L$, were computed by Wentworth (see \cite{PreciseWentwhorthI,WentworthGluing}) and their values are recalled in Section \ref{SectionBoso}. We refer to Section \ref{SectionBoso} for a more precise presentation of this formula. All terms appearing in the previous formula will be defined in Sections \ref{SectionGreen}, \ref{SectionDeterminantOftheLaplacians} and \ref{SectionArakelovMetric}.
 To compute the asymptotic expansion we will use formulae for the asymptotic expansion of the analytic torsion \cite{FINSKI20183457,BismutVasserotAsymptoticAT}. 
The partition function is computed for the Green function (see Section \ref{SectionGreen}) for the canonical metric (see Section \ref{SectionArakelovMetric}) and the integration is made with respect to the Arakelov metric (see Section \ref{SectionArakelovMetric}) using the bosonization formula (see Section \ref{SectionBoso}) in the Section \ref{SectionCalcul}. Moreover, since in the bosonization formula we see that a theta function appears, it will first be easier to compute the modified partition function that we denote $\mathcal{Z}^{(\theta)}$ (see Equation \eqref{ModifPartTheta}) before computing the partition function using an averaging procedure and computations of Section \ref{AnnexeCalculTheta}.

From this asymptotic expansion result, we will deduce convergence of the fluctuations. Let $f: M \to \mathbb{R}$ be a smooth function, then we define: $$\Fluct_{N}f= \sum_{i=1}^{N}f(x_{i})-N\int_{M} \mu_{V} f, $$ with $\mu_{V}$ the equilibrium volume form defined in Equation \eqref{MuV}. Then, proceding like \cite[Section 1.3.1]{ameur2025freeenergyfluctuationsrandom} and \cite[Section 1.3.2]{ameur2025twodimensionalcoulombgasfluctuations}, using the cumulant generating function, we will find the following 

\begin{Corollary}\label{FluctuationsConvergence}
	If $V$ is quasi-subharmonic, and $f : M \to \mathbb{R}$ is a smooth function, then, for the probability measures $\mathbb{P}^{(V)}_{\rho_{can},\rho_{Ar},N}$ given in Equation \eqref{eq:Proba},  $\Fluct_{N}f$ converges in distribution to a normal random variable with mean $m_{f}$ and variance $v_{f}$, where   $$ m_{f} =\frac{1}{8\pi}\int_{M} \mu_{\rho_{can}}f \left(R_{\rho_{can}}-4\pi\chi(M)-\chi(M)\Delta_{\rho_{can}}V+\Delta_{\rho_{can}} \ln (f_{V}) \right)  ~~~~\text{ and } ~~v_{f}=\frac{1}{4\pi}\int_{M}\mu_{\rho_{can}} f\Delta_{\rho_{can}} f >0 $$
	where $\chi(M)=2(1-g)$ is the Euler characteristic of the surface $M$, $R_{\rho_{can}}$ is the scalar curvature for the canonical metric and $f_{V}$ is defined in Equation \eqref{MuV}.
\end{Corollary}
This result, where we consider a Riemann surface without boundary, seems consistent with those obtained in \cite[Theorem 1]{LebléSerfatyFluctuationsCoulombGases}, \cite[Theorem 1.4]{ameur2025twodimensionalcoulombgasfluctuations} where the main difference is the curvature term in the mean of the limit random variable and in our case $f_{V}$ is the density of the equilibrium measure.

\begin{Remark}
	Note that in our case $\Delta$ denotes the positive Laplacian, namely the opposite of the usual Laplacian, therefore the variance term in the previous Corollary is positive. 
\end{Remark}
\subsubsection*{Aknowledgement} We want to thank Siarhei Finski for answering our questions about his results in \cite{FINSKI20183457} which was of great help in obtaining our final result.

\section{Preliminaries}
In this section, we introduce notations and conventions that will be used throughout the article. We will consider $M$ to be a compact Riemann surface of genus $g$, and $\rho=\rho(z)|dz|^{2}$ a metric on $M$. We will denote $\mu_{\rho}=\frac{i}{2}\rho(z)d z \wedge d\bar{z}$ the volume form associated with the metric $\rho$. Since we are on Riemann surfaces, the volume form is also the Kähler form associated with this metric. $\Delta_{\rho}$ will denote the positive Laplace operator for the metric $\rho$ which satisfies $$(\Delta_{\rho} f)dz\wedge d\bar{z} =-4\rho(z)^{-1}\partial\bar\partial f.$$ $R_{\rho}$ will denote the scalar curvature for the metric $\rho$. 
Locally: $$R_{\rho}(z)=\Delta_{\rho} \ln\rho(z).$$ The scalar curvature is twice the Gaussian curvature and from the Gauss-Bonnet Theorem we see that: $$\int_{M}\mu_{\rho}R_{\rho}=8\pi(1-g)=4\pi\chi(M).$$ We will then denote the average scalar curvature for the metric $\rho$ by $\bar{R}_{\rho}=\frac{8\pi(1-g)}{\vol( M,\rho)},$ where $\vol(M,\rho)$ is the volume of the surface $M$ for the metric $\rho$: $$\vol(M,\rho)=\int_{M}\mu_{\rho}.$$ If we consider two metrics $\rho$ and $\rho_{0}$ related by \begin{equation}\label{MetriqueSigma}
    \rho=e^{2\sigma}\rho_{0},
\end{equation}
with $\sigma$ a smooth function on $M$, then the volume form, the scalar Laplacian and the scalar curvature change as: 
\begin{align*}
    \mu_{\rho}=e^{2\sigma}\mu_{\rho_{0}},~~~~
    \Delta_{\rho}=e^{-2\sigma}\Delta_{\rho},~~\text{ and }~~
    R_{\rho}=e^{-2\sigma}(R_{\rho_{0}}+2\Delta_{\rho_{0}}\sigma).
\end{align*}
Moreover we can also note that 
\begin{align}\label{MetriquePhi}
    \mu_{\rho}=\frac{\vol(M,\rho)}{\vol(M,\rho_{0})}\mu_{\rho_{0}}+i\vol(M,\rho)\partial \bar{\partial}\phi
\end{align}
where $\phi$ is such that \begin{equation}
    e^{2\sigma}=\frac{\vol(M,\rho)}{\vol(M,\rho_{0})}-\frac{1}{2}\vol(M,\rho)\Delta_{\rho}\phi.
\end{equation} 

We will also use different classical functionals depending on the metric with $\sigma$ and $\phi$ as in Equations \eqref{MetriqueSigma} and \eqref{MetriquePhi}. 
\begin{Definition}\label{DefLiouvMabAubYau}
    With the previous notation, we define: \begin{itemize}
        \item 
     the \emph{Liouville functional}  
    $$S_{L}(\sigma,\rho_{0})=\int_{M}\mu_{\rho_{0}}(\sigma\Delta_{\rho_{0}}\sigma+R_{\rho_{0}}\sigma),$$

     \item the \emph{Mabuchi functional}  
    $$S_{M}(\sigma,\phi,\rho_{0})=\int_{M}\mu_{\rho_{0}} \left(-2\pi(1-g)\phi\Delta_{\rho_{0}}\phi + \left( \frac{8\pi(1-g)}{\vol(M,\rho_{0})}-R_{\rho_{0}}\right)\phi+\frac{4\sigma e^{2\sigma}}{\vol(M,\rho)}\right),$$
    \item the \emph{Aubin-Yau functional}  
    $$S_{AY}(\phi,\rho_{0})=-\int_{M}\mu_{\rho_{0}}\left(\frac{1}{4}\phi \Delta_{\rho_{0}}\phi - \frac{\phi}{\vol(M,\rho_{0})}\right).$$
    \end{itemize}
\end{Definition}

The preceding functionals will appear in our final result.
\begin{Remark}
    If \( c \in \mathbb{R} \) is a constant, then  
    $$S_{AY}(\phi+c,\rho_{0})=S_{AY}(\phi,\rho_{0})+c.$$
\end{Remark}
All the functionals above satisfy cocycle conditions, namely 
\begin{align*} 
S_{L}(\sigma,\rho_{0})&=-S_{L}(-\sigma,e^{2\sigma}\rho_{0}) &  S_{L}(\sigma_{1}+\sigma_{2},\rho_{0})&=S_{L}(\sigma_{1},\rho_{0})+S_{L}(\sigma_{2},e^{2\sigma_{1}}\rho_{0})\\
    S_{M}(\sigma,\phi,\rho_{0})&=-S_{M}(-\sigma,-\phi,e^{2\sigma}\rho_{0}) & S_{M}(\sigma_{1}+\sigma_{2}, \phi_{1}+\phi_{2},\rho_{0})&=S_{M}(\sigma_{1},\phi_{1},\rho_{0})+S_{M}(\sigma_{2},\phi_{2},e^{2\sigma_{1}}\rho_{0}). 
\end{align*}

In the following we will make use of the bosonization formula and therefore we will introduce a positive line bundle $L$ over $M$ of degree $k=N+g-1.$ In all that follows we will assume that our line bundle can be written as $$L=L_{0}^{k}\otimes E$$ where $L_{0}$ is a degree $1$ line bundle and $E$ is a degree $0$ line bundle. The line bundle $L$ will be endowed with a metric $h.$ This will allow us to define the magnetic field $B(\rho,h)$ for this line bundle, locally: $$B(\rho,h)=\frac{1}{2}\Delta_{\rho} \ln h.$$
If we modify the Hermitian metric by \begin{equation}\label{MetriqueHermPsi}h=e^{-k\psi}h_{0},\end{equation} and the metric $\rho$ as in Equation \eqref{MetriqueSigma}, then the magnetic field will become: \begin{equation}\label{ModifMag}
B(\rho,h)=e^{-2\sigma}\left(B(\rho_{0},h_{0}) -\frac{k}{2}\Delta_{\rho_{0}}\psi\right).
\end{equation} 
Since we want the magnetic field to be positive, we have to require that \begin{equation}\label{eq:CondMagPos}
\Delta_{\rho_{0}}\psi <\frac{2B(\rho_{0},h_{0})}{k}.\end{equation}
The preceding hypotheses will be made in the following.

On this bundle we will consider the Dolbeault operator $\bar{\partial}_{L}: \Omega^{0}(M,L) \to \Omega^{0,1}(M,L).$ The formal adjoint of $\bar{\partial}_{L}$ will be denoted by $\bar{\partial}^{\dagger}_{L}.$

\section{Green functions and Coulomb Gas}\label{SectionGreen}
\subsection{Green functions}
In this article we will be interested in the partition function of the Coulomb gas which involves Green functions.
\begin{Definition}[Green function, {\cite[p.352]{FERRARI-GravitaionMabuchi}}]
Let $(M,\rho)$ a Riemann surface and $\Delta_{\rho}$ the Laplacian associated with the metric $\rho$. Then the \emph {Green function} associated with $\rho$ is the unique function $G_{\rho}$ satisfying $$\Delta_{\rho,x} G_{\rho}(x,y)=-\frac{2\pi \delta(x-y)}{\sqrt{\det \rho}}+\frac{2\pi}{\vol(M,\rho)}$$ and $$\int_{M}\mu_{\rho}(x) G_{\rho}(x,y)=0.$$
\end{Definition}
We could also consider the Arakelov-Green function associated with the metric $\rho$. as in \cite{KlevtsovGenus} and \cite{DHokerPhongFunctionMandelstamDiag}.
\begin{Definition}[Arakelov-Green function associated with a metric, {\cite[eq. (4.38)]{KlevtsovGenus}},{\cite[eq (3.3),(3.4)]{DHokerPhongFunctionMandelstamDiag})}]\label{GreenAr}
Let $(M,\rho)$ a Riemann surface of genus $g\neq1$ with metric $\rho$ and Laplace operator $\Delta_{\rho}$. Then the \emph{Arakelov-Green} function associated with $\rho$ is the unique function $G_{Ar,\rho}$ satisfying $$\Delta_{\rho,x} G_{Ar,\rho}(x,y)=-\frac{2\pi\delta(x-y)}{\sqrt{\det \rho}}+\frac{R_{\rho}(x)}{4(1-g)}$$ and $$\int_{M}\mu_{\rho}(x)R(x)G_{Ar,\rho}(x,y)=0$$
\end{Definition}

The Green-function and the Arakelov-Green function verify the following modifications when changing the metric:

\begin{Proposition}[{\cite[eq. (3.31)]{FERRARI-GravitaionMabuchi}}]\label{VarGreen}
    For $\rho$ and $\rho_{0}$ two metrics as in Equations \eqref{MetriqueSigma} and \eqref{MetriquePhi}: $$G_{\rho}(z,w)-G_{\rho_{0}}(z,w)=-\pi (\phi(z)+\phi(w))+2\pi S_{AY}(\phi,\rho_{0})$$

Similarly for the Arakelov-Green function, we have $$G_{Ar,\rho}(z,w)=G_{Ar,\rho_{0}}(z,w)+\frac{1}{2(1-g)}(\sigma(z)+\sigma(w))-\frac{S_{L}(\sigma,\rho_{0})}{8\pi(1-g)^{2}} .$$
\end{Proposition}
We can also compare the Green function and the Arakelov-Green function using the Ricci potential and the Polyakov functional defined below. 
\begin{Definition}[Ricci Potential and Polyakov functional, {\cite[eq. (3.10)-(3.11), (3.46)]{FERRARI-GravitaionMabuchi}}]\label{PotRicci}  
    The \emph{Ricci potential} \(\Psi\) for the metric \(\rho\) is defined by the following two conditions:  
    $$\Delta_{\rho}\Psi(x,\rho)=R_{\rho}(x)-\bar{R}_{\rho} ~~\text{ and }~~\int_{M}\mu_{\rho}(x) \Psi(x,\rho)=0.$$  
    where $R_{\rho}$ is the scalar curvature and $\bar{R}_{\rho}$ is the average scalar curvature.  
The \emph{Polyakov functional} is defined by:  
    $$\Psi_{P}(\rho)=-\frac{1}{8\pi}\int_{M^{2}}\mu_{\rho}^{2}(x,y)R_{\rho}(x)R_{\rho}(y)G_{\rho}(x,y)=\frac{1}{4}\int_{M}\mu_{\rho}(x)R_{\rho}(x)\Psi(x,\rho).$$  
\end{Definition}

\begin{Proposition}[{\cite[eq. (3.45)]{FERRARI-GravitaionMabuchi}}]  \label{Pr: RicciGreen}
    The Ricci potential \(\Psi\) can be written as:  
    $$\Psi(x,\rho)=-\frac{1}{2\pi}\int_{M}\mu_{\rho}(y)R_{\rho}(y)G_{\rho}(x,y).$$  
\end{Proposition}
We can also compute how the Ricci potential and the Polyakov functional change when modifying the metric:
\begin{Proposition}[{\cite[eq. (3.41), (3.46)]{FERRARI-GravitaionMabuchi}}]\label{RicVar}
    We have, with $\rho$ and $\rho_{0}$ as in equations \eqref{MetriqueSigma} and \eqref{MetriquePhi}:  
    $$\Psi(x,\rho)-\Psi(x,\rho_{0})=2\sigma(x)+4\pi(1-g)\left(\phi(x)-S_{AY}(\phi,\rho_{0}) \right)-\frac{1}{2}S_{M}(\sigma,\phi,\rho_{0})$$ 
    and for the Polyakov functional:  
    $$\Psi_{P}(\rho)-\Psi_{P}(\rho_{0})=S_{L}(\sigma,\rho_{0})-2\pi(1-g)S_{M}(\sigma,\phi,\rho_{0}),$$
    where $S_{L}$, $S_{M}$ and $S_{AY}$ are defined in Definition \ref{DefLiouvMabAubYau}.
\end{Proposition}
    We can observe that, if $R_{\rho}$ is constant, then $\Psi(x,\rho)=0$ and $\Psi_{P}(\rho)=0.$
    
Finally, we have the relation between the Green function and the Arakelov-Green function. 
\begin{Proposition}
    If $g\neq1$, we have the following equality: 
    $$G_{Ar,\rho}(x,y)=G_{\rho}(x,y)+\frac{\Psi(x,\rho)+\Psi(y,\rho)}{4(1-g)}-\frac{\Psi_{P}(\rho)}{8\pi(1-g)^{2}}.$$
\end{Proposition}

\begin{proof}
    We compute: \begin{align*}
        \Delta_{\rho,x}\left(G_{\rho}(x,y)+\frac{\Psi(x,\rho)+\Psi(y,\rho)}{4(1-g)}-\frac{\Psi_{P}(\rho)}{8\pi(1-g)^{2}}\right)&=\Delta_{\rho,x}G_{\rho}(x,y)+\frac{1}{4(1-g)}\Delta_{\rho,x}\Psi(x,\rho)\\&=-\frac{2\pi\delta(x-y)}{\sqrt{\det\rho}}+\frac{2\pi}{\vol(M,\rho)}+\frac{R_{\rho}(x)}{4(1-g)}-\frac{\bar{R}_{\rho}}{4(1-g)}\\&=-\frac{2\pi\delta(x-y)}{\sqrt{\det\rho}}+\frac{R_{\rho}(x)}{4(1-g)}.
    \end{align*}
    From the Definition of the Polyakov functional, we can then easily see that the integral with the scalar curvature of the right hand side in the Proposition is $0$, thus proving the result.
\end{proof}

Consequently, we also have for $g \neq 1$: 
\begin{equation}\label{Eq:GarRho->GRho0}
    G_{Ar,\rho}(x,y)=G_{\rho_{0}}(x,y)+\frac{2\sigma(x)+\Psi(x,\rho_{0})+2\sigma(y)+\Psi(y,\rho_{0})}{4(1-g)}-\frac{\Psi_{P}(\rho_{0})+S_{L}(\sigma,\rho_{0})}{8\pi(1-g)^{2}}.
\end{equation}

\subsection{Coulomb gas}
The Coulomb gas we will be interested in will be defined to be a system of $N$ interacting particles on a compact Riemann surface $M$ whose interactions will be given by a Green function and with a potential $-\frac{V}{2}.$ More precisely, the system of $N$ particles follows the Hamiltonian:  
$$H_{\rho,N}(V)=-\sum_{1\leq i <j \leq N} G_{\rho}(z_{i},z_{j})-\frac{N+g-1}{2}\sum_{i=1}^{N} V(z_{i}).$$ The constant $N+g-1$ will be denoted by $k$ and will correspond to the degree of some line bundle over $M$. This choice is particularly convenient but not strictly necessary. The potential $-V$ that we will consider may depend on $k$ and should satisfy uniformly: $$V(x)=V_{0}(x)+O\left(\frac{1}{k}\right)$$ when $k\to +\infty$. The potential $V$ will be supposed smooth.
The partition function will then be defined as follow: 

\begin{Definition}\label{DefFoncPart}
    We define the $N$-particles \emph{partition function} for the metrics $\tilde{\rho}$ and $\rho$ for the Coulomb gas by $$\mathcal{Z}_{\rho,\tilde{\rho},N}(V)=\frac{1}{N!}\int_{M^{N}}\mu_{\tilde{\rho}}^{N}\exp\left(2\sum_{1\leq i < j \leq N} G_{\rho}(p_{i},p_{j}) +k\sum_{i=1}^{N} V(p_{i})\right) $$
with $k=N+g-1.$
\end{Definition}

\begin{Remark} 
    We will note $$\mathcal{Z}_{\rho,N}(V)=\mathcal{Z}_{\rho,\rho,N}(V)$$ and, we can observe that $$\mathcal{Z}_{\rho_{1},\rho_{0},N}(V)=\mathcal{Z}_{\tilde{\rho}_{1},\tilde{\rho}_{0},N}(\tilde{V})$$ with $$\tilde{V}=V-\frac{2\tilde{\sigma}_{0}}{k}-2\pi \left(1-\frac{g}{k} \right)\left(S_{AY}(\tilde{\phi}_{1},\rho_{1})-\tilde{\phi}_{1} \right)$$ if $\tilde{\rho}_{0}$ and $\rho_{0}$, and $\tilde{\rho}_{1}$ and $\rho_{1}$ are related by  $$\tilde{\rho}_{0}=e^{2\tilde{\sigma}_{0}}\rho_{0}~~ \text{ and } ~~ \mu_{\tilde{\rho}_{1}}=\frac{\vol(M,\tilde{\rho}_{1})}{\vol(M,\rho_{1})}\mu_{\rho_{1}}+i\vol(M,\tilde{\rho}_{1})\partial\bar{\partial}\tilde{\phi_{1}}$$ as in Equations \eqref{MetriqueSigma} and \eqref{MetriquePhi}.
\end{Remark}
We could also define the partition function for the Coulomb gas using the Arakelov-Green function, but this would be equivalent to just changing the potential $V$ using Equation \eqref{Eq:GarRho->GRho0}, therefore we will perform the computation with the usual Green function.
The aim of this article will be to find an asymptotic expansion of $\ln\mathcal{Z}_{\rho,\tilde{\rho},N}(V)$ (see Theorem \ref{ThéorèmeFinalSansTheta}).
\section{Determinants of the Laplacians}\label{SectionDeterminantOftheLaplacians}
Considering $M$ a compact Riemann surface of genus $g$, endowed with a metric $\rho$ we can denote $\Delta_{\rho}$ as before to be the positive scalar Laplacian. The spectrum of $\Delta_{\rho}$ is positive and discrete, and we can then do a regularization using zeta functions to define its determinant that we will note $\det\Delta_{\rho}$ 
. If we consider two metrics as in Equation \eqref{MetriqueSigma}: 
\begin{Proposition}[{\cite[Equation (1.13)]{OSGOOD1988148}}]\label{VarDetSca}
    The determinant of the scalar Laplace operator satisfies  
    \begin{align*}
        \ln\det\Delta_{\rho}
        &=-\frac{S_{L}(\sigma,\rho_{0})}{12\pi}+\ln\bigg(\frac{\vol(M,\rho)}{\vol(M,\rho_{0})}\bigg)+\ln\det\Delta_{\rho_{0}},
         \end{align*}
         with $S_{L}$ the Liouville functional (see Definition \ref{DefLiouvMabAubYau}).
\end{Proposition}

In particular, if $\sigma$ is constant, we have $\rho=\alpha \rho_{0}$, with $\alpha>0$ and: 
\begin{equation}\label{VarDetLapScaCste}
    \ln\det \Delta_{\alpha \rho_{0}}=\left(1-\frac{\chi(M)}{6} \right)\ln \alpha+\ln\det\Delta_{\rho_{0}}
\end{equation}
where $\chi(M)=2(1-g)$ is the Euler-characteristic of $M.$

We will now consider $L \to M$ to be a Hermitian line bundle over $M$ with Hermitian metric $h.$ In the following we will denote by $H^{0}(M,L)$ the set of holomorphic sections of this bundle and $(\omega_{i})_{i}$ a basis of holomorphic sections. The magnetic Laplacian will be defined to be $\square_{L}=2\bar\partial^{\dagger}_{L}\bar{\partial}_{L}.$ This determinant, that we will call the magnetic Laplacian is often called the Kodaira Laplacian or the Dolbeault Laplacian.
Then, we can define the determinant of the magnetic Laplacian $\square_{L,\rho,h}$. This is related to the Ray-Singer analytic torsion \cite{RAY1971145} and the Quillen metric \cite{Quillen1985DeterminantsOC,BismutGilletSouleQuillenMetric}, see \cite{HolomorphicMorseMaMarinescu} for review. Indeed, if $T$ is the analytic torsion as in \cite{FINSKI20183457} or \cite{HolomorphicMorseMaMarinescu}, then we have $\ln\det\left(\frac{1}{2}\square_{L,\rho,h}\right)=2\ln T(\rho,h)$. Modifying the metrics as in Equation \eqref{MetriqueSigma} and \eqref{MetriqueHermPsi}, we have formulas for the variation of $\ln \det \square_{L,\rho,h}.$ First, we have to introduce some functionals: 

\begin{Definition}\label{DefFonctionnelleS}
    We define: $$S_{1}(\sigma,\psi,\rho_{0},h_{0})=\frac{1}{2\pi}\int_{M}\mu_{\rho_{0}}\left(-\frac{1}{2}\psi R_{\rho_{0}} +2\sigma\left(\frac{B(\rho_{0},h_{0})}{k}-\frac{1}{2}\Delta_{\rho_{0}}\psi \right) \right)$$
    and $$S_{2}(\psi,\rho_{0},h_{0})=\frac{1}{2\pi}\int_{M}\mu_{\rho_{0}}\left(-\frac{1}{4}\psi\Delta_{\rho_{0}}\psi+\frac{B(\rho_{0},h_{0})}{k}\psi\right).$$
    We will often omit some arguments in these functionals when the context is clear.
\end{Definition}
\begin{Remark}
    Like for the previous functionals, we have cocycle properties: \begin{align*}
        S_{1}(\sigma,\psi,\rho_{0},h_{0})&=-S_{1}(-\sigma,-\psi,\rho_{0}e^{2\sigma},h_{0}e^{-k\psi}) & S_{2}(\psi,\rho_{0},h_{0})&=-S_{2}(-\psi,\rho_{0},h_{0}e^{-k\psi}).
    \end{align*}
    Moreover, if $c \in \mathbb{R}$ is a constant real number, then: \begin{align*}
        S_{1}(\sigma,\psi+c,\rho_{0},h_{0})&=S_{1}(\sigma,\psi,\rho_{0},h_{0})+2(g-1)c  
        \end{align*}
        and \begin{align*}S_{2}(\psi+c,\rho_{0},h_{0})&=S_{2}(\psi,\rho_{0},h_{0})+\frac{c}{2\pi k}\int_{M}\mu_{\rho_{0}}B(\rho_{0},h_{0})=S_{2}(\psi,\rho_{0},h_{0})+c.
    \end{align*}
\end{Remark}

An important point for our computations is knowing how the different quantities that will appear transform when we change the different metrics. For this reason we recall some results that can be found in \cite{KlevtsovQuillen}.

\begin{Proposition}[{\cite[Th. 4]{KlevtsovQuillen}}]\label{VarDetFi}
    The following exact equality holds for metrics as in Equations \eqref{MetriqueSigma} and \eqref{MetriqueHermPsi} \begin{align*}
    \ln\frac{\det\langle \omega_{i},\omega_{j}\rangle_{\rho,h}}{\det\square_{L,\rho,h}}-\ln\frac{\det\langle \omega_{i},\omega_{j}\rangle_{\rho_{0},h_{0}}}{\det\square_{L,\rho_{0},h_{0}}}&=-k^{2}S_{2}(\psi,\rho_{0},h_{0})+\frac{k}{2}\left(\ln \left(\frac{\vol(M,\rho_{0})}{\vol(M,\rho)} \right)+S_{1}(\sigma,\psi,\rho_{0},h_{0})  \right)\\& +\frac{1-g}{3}\ln \left(\frac{\vol(M,\rho_{0})}{\vol(M,\rho)} \right)+\frac{1}{12\pi}S_{L}(\sigma,\rho_{0}).
    \end{align*}
\end{Proposition}
\begin{Remark}In the previous proposition we added the term $\ln \left(\frac{\vol(M,\rho_{0})}{\vol(M,\rho)} \right)$, not present in \cite{KlevtsovQuillen}, where it was set $\vol(M,\rho)=\vol(M,\rho_{0})=2\pi$. 
\end{Remark}

Now, we want to know the transformation of $\ln\det\square_{L,\rho,h}$ when changing $\rho$ and $h$ as in Equations \eqref{MetriqueSigma} and \eqref{MetriqueHermPsi}. The terms $\mathcal{F}(\rho,h)$ defined below will appear.
\begin{Definition}
    We define for $B(\rho,h)$ a positive magnetic field
    $$\mathcal{F}(\rho,h)=\frac{1}{2\pi}\int_{M}\mu_{\rho} \left( \frac{B(\rho,h)}{2}\ln \frac{B(\rho,h)}{2\pi}+\frac{R_{\rho}}{6}\ln \frac{B(\rho,h)}{2\pi}+\frac{\ln B(\rho,h)}{24}\Delta_{\rho} \ln B(\rho,h )\right).$$
\end{Definition}

Consequently, for metrics as in Equations \eqref{MetriqueSigma} and \eqref{MetriqueHermPsi} we have the following result:
\begin{Proposition}[{\cite[Th. 4]{KlevtsovQuillen}}]\label{FormuleVarLapMag}
     We have for $B(\rho,h)$ and $B(\rho_{0},h_{0})$ positive: $$\ln \det \square_{L,\rho,h}=\ln \det \square_{L,\rho_{0},h_{0}}+\mathcal{F}\left(\rho_{0},h_{0} \right)-\mathcal{F}\left(\rho,h\right)+O\left(\frac{1}{k} \right).$$
\end{Proposition}
We will make use of the two preceding Propositions in our computation.
\begin{Remark}
    In our case, we will have $L=L_{0}^{k}\otimes E$ and $B(x)=kf(x)$ consequently: 
    \begin{align*}
    \mathcal{F}(\rho,h)&=\frac{k\ln k}{4\pi}\int_{M}\mu_{\rho}f + \frac{k}{4\pi}\left(\int_{M}\mu_{\rho} f \ln(f)-\ln2\pi\int_{M}\mu_{\rho} f\right)+\frac{2}{3}(1-g)\ln k\\&+\frac{1}{12\pi}\int_{M}\mu_{\rho}R_{\rho}\ln(f)+\frac{1}{48\pi}\int_{M}\mu_{\rho}\ln(f)\Delta_{\rho}\ln(f)-\frac{2}{3}(1-g)\ln2\pi.
    \end{align*}
    Moreover, we will have $$\int_{M} \mu_{\rho} f=2\pi,$$ hence: \begin{align*}
        \mathcal{F}(\rho,h)=\frac{k\ln k}{2}+\frac{k}{2}\left(\int_{M}\mu_{\rho}\frac{f}{2\pi}\ln(f)-\ln2\pi \right)+\frac{1}{12\pi}\int_{M}\mu_{\rho}R_{\rho}\ln (f)+\frac{2}{3}(1-g)\ln(k)+\frac{1}{48\pi}\int_{M}\mu_{\rho}\ln(f)\Delta_{\rho}\ln(f)-\frac{2}{3}(1-g)\ln2\pi
    \end{align*}
    and with $b=\frac{f}{2\pi}=\frac{B}{2\pi k}$: $$\mathcal{F}(\rho,h)=\frac{k\ln k}{2}+\frac{k}{2}\int_{M}\mu_{\rho}b\ln(b)+\frac{2}{3}(1-g)\ln(k)+\frac{1}{12\pi}\int_{M}\mu_{\rho}R_{\rho}\ln(b)+\frac{1}{48\pi}\int_{M}\mu_{\rho}\ln(b)\Delta_{\rho}\ln(b).$$
\end{Remark}
Another interesting point about the transformation of the magnetic Laplacian is knowing how it transforms when rescaling the metric by a constant.
\begin{Proposition}[{\cite[eq. (2.39)]{Fay1992KernelFA}}]\label{ConstantLapMag}
    Let $\det \square_{\rho,h}$ be the determinant of the magnetic Laplacian on the Riemann surface $M$ of genus $g$ for the metric $\rho$. Then, if $\alpha > 0$, we have:
    \[
    \det \square_{\alpha \rho,h} = \alpha^{h^{0}(L) - \frac{k}{2} - \frac{1}{3}(1 - g)} \det \square_{\rho,h}.
    \]
    In the case of interest, we have $h^{0}(L) = k + 1 - g$, so:
    \[
    \det \square_{\alpha \rho,h} = \alpha^{\frac{k}{2} + \frac{2}{3}(1 - g)} \det \square_{\rho,h},
    \]
    hence
    \[
    \ln \det \square_{\alpha \rho,h} = \ln \det \square_{\rho,h} + \ln \alpha \left[ \frac{k}{2} + \frac{2(1 - g)}{3} \right].
    \]
\end{Proposition}
An interesting fact in the previous Proposition is that the formula is exact, whereas in Proposition \ref{FormuleVarLapMag}, we have a remainder $O\left(k^{-1} \right)$.

Now we recall here the result of Finski \cite{FINSKI20183457} that we will need for our computation. We give below the results necessary for our computation with the notations that we introduced previously for the case of a Riemann surface. Our notations are different to that used in \cite{FINSKI20183457}, and in particular we suppose that the bundle $E$ appearing in \cite{FINSKI20183457} is trivial. 
\begin{theorem}[see {\cite[Th. 1.1 and 1.3]{FINSKI20183457}}]\label{ThFinski}
    The determinant of the magnetic Laplacian for the line bundle $L$ over $M$ admits an asymptotic expansion of the form:
    $$\ln \det \square_{L,\rho,h}=\alpha_{0}k\ln(k)+\beta_{0}k+\alpha_{1}\ln(k)+\beta_{1}+O\left(\frac{\ln k}{k}\right)$$
    where $$\alpha_{0}=-\frac{1}{2}\int_{M} \frac{B(\rho,h)}{2\pi k}\mu_{\rho} ~~ \text{ and } ~~ \beta_{0}=\frac{1}{2}\int_{M}\ln \left( \frac{B(\rho,h)}{2\pi k}\right)\frac{B(\rho,h)}{2\pi k}\mu_{\rho}-\frac{\ln2}{2}$$ when the degree $k$ goes to $+\infty.$ Moreover, if $B=2\pi k$ (i.e. $b=1)$, then: $$\alpha_{1}=\frac{2(g-1)}{3} ~~ \text{ and } ~~ \beta_{1}=\frac{g-1}{12}\left(24\zeta'(-1)+2\ln2\pi+7 +8\ln2\right),$$
    with $\zeta$ the Riemann zeta function.
\end{theorem}
Note that the values of $\beta_{0}$ and $\beta_{1}$ differ from that in \cite{FINSKI20183457}, indeed the terms $\frac{\ln2}{2}$ and $8\ln2$ come from the fact that we consider $\square_{L,\rho,h}=2\bar\partial_{L}^{\dagger}\bar{\partial}_{L}$ instead of $\bar{\partial}_{L}^{\dagger}\bar{\partial}_{L}.$
\begin{Remark}
    If $B(\rho,h)=2\pi k$, then $$\mathcal{F}(\rho,h)=\frac{1}{2}\vol(M,\rho)k\ln k+\frac{2}{3}(1-g)\ln k.$$
\end{Remark}
\begin{Remark}\label{UniformeFinski}
In the previous theorem, the values of $\alpha_{i}$ and $\beta_{i}$ will only depend on the magnetic field $B$ when varying the metric and moreover, this expansion will be uniform when varying the line bundle in the Jacobian (see \cite[Remark 1.2 and Equation (3.54)]{FINSKI20183457}), and this will be useful in Section \ref{SectionFinal}. 
\end{Remark}

\section{Arakelov Metric}\label{SectionArakelovMetric}
In this Section we will introduce the Arakelov metric which appears in the bosonization formula and give some properties of this metric.
\subsection{Definition of Arakelov metric}
Before defining the Arakelov metric, we introduce some notations on Riemann surfaces and the canonical metric, which are standard (see \cite{DHokerPhongFunctionMandelstamDiag,KlevtsovGenus,WentworthArakelov}).
\begin{Definition}
    Let $M$ be a compact Riemann surface of genus $g>0$. Denote by $(A_j, B_j)_{1 \leq j \leq g}$ a basis of $H_1(M, \mathbb{Z})$ whose intersection numbers satisfy:
    \[ \forall j,\ell, \quad
    A_j \circ B_\ell = \delta_{j\ell}, \quad A_j \circ A_\ell = 0, \quad B_j \circ B_\ell = 0.
    \]
    We then denote by $(\omega_i)_{1 \leq i \leq g}$ the normalized basis of Abelian differentials satisfying:
    \[ \forall j,\ell, \quad
    \int_{A_j} \omega_\ell = \delta_{j\ell}.
    \]
    We then define the matrix $\tau$ with entries
    \[
    \tau_{ij} = \int_{B_i} \omega_j
    \]
    as the \emph{period matrix}.
\end{Definition}

    The matrix $\tau$ is symmetric, and $\mathfrak{Im}(\tau)$ is positive definite.
We can use this matrix to define a lattice: 
$$\Lambda_{\tau}=\{m+\tau n | m,n \in \mathbb{Z}^{g}\}\subseteq \mathbb{C}^{g}$$ and the torus $$\Jac(M)=\mathbb{C}^{g}/\Lambda_{\tau}$$ called the Jacobian torus. 
One important fact about the Jacobian torus is that we can define the Abel map for a fixed base point $z_{0}$: $$z \in M \mapsto \left(\int_{z_{0}}^{z} \omega_{i} \right)_{1\leq i\leq g}\in \Jac(M) $$ which is an holomorphic embedding. On $\mathbb{C}^{g}$, we can define the theta function: 
$$\theta(z,\tau)=\sum_{n\in \mathbb{Z}^{g}}\exp\left(i\pi n^{T}\tau n+2i\pi n^{T}z \right).$$
In addition, we will denote by $\mathcal{D}$ the vector of Riemann constants: $$\mathcal{D}_{j}=\frac{1-\tau_{jj}}{2}+\sum_{i\neq j}\int_{A_{i}}\omega_{i}(z)\int_{z_{0}}^{z}\omega_{j}.$$
We will now define a normalized metric on the Riemann surface which appears naturally and is called the canonical metric.
\begin{Definition}\label{Def:MétCan}
    Let $M$ be a compact Riemann surface of genus $g>0$, and let $\tau_{ij}$ be the period matrix. The \emph{canonical metric} is defined by
    \[
    \mu_{\rho_{can}} = \frac{i}{2g} \sum_{i,j=1}^{g} (\mathfrak{Im} \tau)_{ij}^{-1} \omega_i \wedge \bar{\omega}_j.
    \]
    Observe that we have $\int_M \mu_{\rho_{can}} = 1$.
\end{Definition}
We can observe that the canonical metric is the metric induced by the flat metric on the Jacobian torus via the Abel map.

\begin{Remark}
    The previous definition does not apply to the case of genus $g=0$. In this case, we set
    \[
    \mu_{\rho_{can}} = \frac{i}{2\pi} \frac{dz \wedge d\bar{z}}{(1 + |z|^2)^2}.
    \]
    We can note that this is just the usual Fubini-Study metric rescaled so that $\int_{M}\mu_{\rho_{can}}=1.$
\end{Remark}

\begin{Definition}[\cite{ArakelovSJ_1974,WentworthArakelov,WentworthGluing}]
    The \emph{Arakelov-Green} function $G_{Ar}(x,y)$ on $M$ of genus $g$ is characterized by the following properties:
    \begin{enumerate}
        \item $G_{Ar}(x, y) = G_{Ar}(y, x)$,
        \item $\exp G_{Ar}(x, y)$ has a zero of order 1 at $x = y$,
        \item $\partial_z \bar{\partial}_z G_{Ar}(z, w) = i \pi \mu_{\rho_{can}}(z)$ for $z \neq w$,
        \item $\int_M \mu_{\rho_{can}}(z) G_{Ar}(z, w) = 0$ for all $w \in M$.
    \end{enumerate}
    \begin{Remark}
        The function that we denote by $G_{Ar}$ is usually denoted by $\ln G_{Ar}$ in Arakelov literature.
    \end{Remark}
\end{Definition}
 \begin{Remark}  
 By Arakelov-Green function we mean the Arakelov-Green function for the Arakelov metric. Arakelov-Green function equals the usual Green function for the canonical metric: \begin{equation}\label{Eq:GAR=Gcan}
  G_{Ar}=G_{Ar,\rho_{Ar}}=G_{\rho_{can}}.
   \end{equation}
\end{Remark}

\begin{Proposition}\cite[p.1176]{ArakelovSJ_1974},\cite[(B.1) p.457]{WentworthArakelov}, \cite[p.21]{Fay1992KernelFA}
    For the case of the sphere, we have:
    \[
    G_{Ar}^{\mathbb{S}^{2}}(z, w) = \frac{1}{2}+\ln\left(\frac{ |z - w|}{\sqrt{(1 + |z|^2)(1 + |w|^2)}}\right).
    \]
    For the case of the torus the Arakelov-Green function reads
    \[
    G_{Ar}^{\mathbb{T}^{2}_{\tau}}(z, w) = \ln\left| \frac{\theta_1(z - w, \tau)}{\eta(\tau)} \right| +  \frac{\pi}{4 \mathfrak{Im}(\tau)} (z - w - \bar{z} + \bar{w})^2.
    \]
                where $\eta$ is the Dedekind eta function $\eta(\tau)=q^{\frac{1}{24}}\prod_{n=1}^{\infty}(1-q^{n})$ with $q=e^{2i\pi \tau}$.
        The function $\theta_{1}$ is the first Jacobi theta function: 
    $$\theta_{1}(z,\tau)=i\sum_{n\in\mathbb{Z}}(-1)^{n+1}q^{\left(n+\frac{1}{2}\right)^{2}} e^{2i\pi nz}$$
satisfying $\theta_{1}(0,\tau)=0.$
\end{Proposition}
From the Arakelov-Green function, we can define the Arakelov metric which is of special importance for our computation. Indeed, most of the quantities appearing in the bosonization formula are computed for this metric and for this reason we will do our computation of the partition function for this metric.
\begin{Definition}[\cite{ArakelovSJ_1974,WentworthArakelov,WentworthGluing}]\label{ArMét}
    We define the \emph{Arakelov metric} $\rho_{Ar} = \rho_{Ar}(z) |dz|^2$ with
    \[
    \ln \rho_{Ar}(z) = 2 \lim_{w \to z} \left( G_{Ar}(z, w) - \ln |z - w| \right).
    \]
\end{Definition}
It is well known and  can be easily computed for the sphere and the torus. 
\begin{Proposition}\cite[(B.1) p.457]{WentworthArakelov}, \cite[p.21]{Fay1992KernelFA}. 
    For the sphere:
    \[
    \rho_{Ar}^{\mathbb{S}^{2}}(z) = \frac{e}{(1 + |z|^2)^2}
    \]
    and for the torus:
    \[
    \rho_{Ar}^{\mathbb{T}^{2}_{\tau}}(z) = 4\pi^2 |\eta(\tau)|^4.
    \]
\end{Proposition}
Those two metrics, for the sphere and the torus, are simply constant rescalings of the usual metrics on the sphere and the torus. However, this is not the case for $g>1$.
We can easily observe that we have $\Ric(\rho_{Ar}^{\mathbb{S}^{2}}) = 4\pi \mu$ and $\Ric(\rho_{Ar}^{\mathbb{T}^{2}_{\tau}}) = 0 = -4\pi (g-1) \mu$, which means (see below) that the Arakelov metric seen as a Hermitian metric on the anticanonical bundle for the torus and the sphere is admissible. This is actually true for every genus $g$.

The volumes of the sphere and the torus for the Arakelov metric can be explicitly computed. 
\begin{Proposition}\label{Prop:VolumeArakelov}
    We have: $$\mu_{\rho_{Ar}^{\mathbb{S^{2}}}}=\frac{ie}{2(1+|z|^{2})}dz\wedge d\bar{z} ~~ \text{ and } ~~ \mu_{\rho_{Ar}^{\mathbb{T^{2}_{\tau}}}}=2i\pi^{2}|\eta(\tau)|^{4}dz\wedge d\bar{z}.$$ Hence 
    \[
    \vol(\mathbb{S}^2,\rho_{Ar}) = \pi e
    ~~
    \text{ and } ~~
    \vol(\mathbb{T}^2_{\tau},\rho_{Ar}) = 4\pi^2 \mathfrak{Im}(\tau) |\eta(\tau)|^4.
    \]
\end{Proposition}

For the case of genus $g>1$, we do not have such simple formulas but it is possible to express quantities about the Arakelov metric using hyperbolic metrics (see Appendix \ref{AnnexeArakelovHyperbolic}).
\subsection{Admissible Metric}
To make use of the bosonization formula we need to consider an admissible Hermitian metric.
\begin{Definition}\label{Def:MétAdm} (\cite[eq. (4.3)]{WentworthGluing}, \cite[p.17]{Fay1992KernelFA})
    Let $L \to M$ be a line bundle of degree $k$. A Hermitian metric $h$ on $L$ is said to be \emph{admissible} if
    \[
    \Ric(h) = 2\pi k \mu_{\rho_{can}}.
    \]
\end{Definition}
In particular, we can note that the Arakelov metric can be seen as an admissible metric on the anti-canonical bundle.
\begin{Proposition}[{\cite[eq. (4.4)]{WentworthGluing}}, {\cite[eq. (1.33)]{Fay1992KernelFA}}]
    Let $\rho_{Ar}$ be the Arakelov metric. We have
    \[
    \Ric(\rho_{Ar}) = -i\partial \bar{\partial} \ln(\rho_{Ar}) = -4\pi (g-1) \mu_{\rho_{can}}=2\pi\chi(M)\mu_{\rho_{can}}.
    \]
\end{Proposition}

An important fact is that we can always define a line bundle of degree $k$ with admissible metric (see \cite{FaltingsCalculusArithmeticSurface1984,Fay1992KernelFA}).

\subsection{Relations between Arakelov metric and canonical metric}\label{Section:ArakelovCanonical}
We can always note $\rho_{Ar}=e^{2\sigma_{Ar}}\rho_{can}$, where $\rho_{can}$ is the canonical metric. With this notation, we can observe that $$R_{\rho_{Ar}}=8\pi(1-g)e^{-2\sigma_{Ar}} ~~ \text{ and } ~~ B(\rho_{Ar},h)=2\pi k e^{-2\sigma_{Ar}}$$ for $h$ and admissible metric. 
Moreover, from Equations \eqref{Eq:GarRho->GRho0} and \eqref{Eq:GAR=Gcan} we can compute: $$\sigma_{Ar}(x)=-\frac{\Psi(x,\rho_{can})}{2}+\frac{\Psi_{P}(\rho_{can})-S_{L}(-\sigma_{Ar},\rho_{Ar})}{8\pi(1-g)}.$$
Similarly, we can define $\phi_{Ar}$ relating $\mu_{\rho_{Ar}}$ and $\mu_{\rho_{can}}$ as in Equation (\ref{MetriquePhi}). Then, for $g\neq1$: $$\phi_{Ar}(x)=\frac{\Psi(x,\rho_{Ar})}{4\pi(1-g)}-\int_{M}\mu_{\rho_{can}}\phi_{Ar}+2S_{AY}(\phi_{Ar},\rho_{can})$$ and since $\phi_{Ar}$ is defined up to a constant, we can take, for $g\neq1$: \begin{equation}\label{PhiHat}\hat{\phi}_{Ar}(x)=\frac{\Psi(x,\rho_{Ar})}{4\pi(1-g)}.\end{equation}
\begin{Remark}
    For $g=0,1$, we can take $\hat{\phi}_{Ar}=0,$ since in those cases the Arakelov metric and the canonical metric are constant rescalings of the usual metrics.
\end{Remark}
With the previous notation, the following computations hold: 
\begin{equation}\label{Eq:IntSigmaArDelta-PsiP}\int_{M} \mu_{Ar} \sigma_{Ar} \Delta_{Ar} \sigma_{Ar}=\Psi_{P}(\rho_{can})\end{equation} and $$\Psi_{P}(\rho_{can})-S_{L}(-\sigma_{Ar},\rho_{Ar})=\Psi_{P}(\rho_{Ar})-2\pi(1-g)S_{M}(-\sigma_{Ar},-\phi_{Ar},\rho_{Ar})$$ and $$\Psi_{P}(\rho_{Ar})=16\pi^{2}(1-g)^{2}S_{AY}(\hat{\phi}_{Ar},\rho_{can}).$$ Hence: \begin{equation}\label{Eq:IntSigmaAr-PsiP-SAY}\int_{M} \mu_{\rho_{can}}\sigma_{Ar}=\frac{\Psi_{P}(\rho_{can})-S_{L}(-\sigma_{Ar},\rho_{Ar})}{8\pi(1-g)}=2\pi(1-g)S_{AY}(\hat{\phi}_{Ar},\rho_{can})-\frac{1}{4}S_{M}(-\sigma_{Ar},-\phi_{Ar},\rho_{Ar})\end{equation}

Now, for the special cases where $g=0,1$, the previous quantities can be easily computed.
\begin{itemize}
    \item For $g=0$, we have $\sigma_{Ar}=\frac{1+\ln\pi}{2}$ and $\Psi(x,\rho_{Ar})=\Psi(x,\rho_{can})=\Psi_{P}(\rho_{can})=\Psi_{P}(\rho_{Ar})=0$, therefore $$S_{L}(-\sigma_{Ar},\rho_{Ar})=-4\pi(1+\ln\pi) ~~ \text{ and }~~ S_{M}(-\sigma_{Ar},-\phi_{Ar},\rho_{Ar})=-2(1+\ln \pi).$$
\item Similarly, for $g=1$, $\sigma_{Ar}=\ln \big(2\pi \sqrt{\mathfrak{Im}(\tau)} |\eta(\tau)|^{2}\big)$ and $\Psi(x,\rho_{Ar})=\Psi(x,\rho_{can})=\Psi_{P}(\rho_{can})=\Psi_{P}(\rho_{Ar})=0$, moreover $S_{L}(-\sigma_{Ar},\rho_{Ar})=0$, and $\int_{M} \mu_{\rho_{can}}\sigma_{Ar}=\sigma_{Ar}$ and we can compute $$S_{M}(-\sigma_{Ar},-\phi_{Ar},\rho_{Ar})=-4\sigma_{Ar}=-\ln\left(16\pi^{4}\mathfrak{Im}(\tau)^{2}|\eta(\tau)|^{8} \right).$$
\end{itemize}

In the case where $g\geq2$, the relations between the canonical metric, the Arakelov metric and the hyperbolic metric of constant scalar curvature can be expressed using quantities of hyperbolic geometry (see Appendix \ref{AnnexeArakelovHyperbolic}).

\section{Bosonization formula}\label{SectionBoso}
In order to obtain the asymptotic expansion when $N$ tends to infinity of the partition function of the Coulomb gas, we will use the bosonization formula relating the determinant of the magnetic Laplacian on a Hermitian line bundle $(L,h)$ of degree $k=N+g-1$ and different quantities for the Arakelov metric including the determinant of the scalar Laplacian and Arakelov-Green functions appearing in the definition of the partition function for the Coulomb gas.
\begin{Proposition}[Bosonization formula, see {\cite[eq (1.1)]{WentworthGluing}},{\cite[Th. 5.11]{Fay1992KernelFA}}]\label{Pr:FormuleBoso}
    If $k \geq g$, for $h$ an admissible metric on a line bundle $L$ of degree $k=N+g-1$, we have $$\frac{\det \square_{L,\rho_{Ar},h}}{\det \langle \omega_{i}, \omega_{j} \rangle_{\rho_{Ar},h}}= B_{g,k} \frac{\exp\left(2\sum_{1\leq i <j \leq N} G_{Ar}(p_{i},p_{j}) \right)}{\| \det \omega_{i}(p_{j})\|_{h}^{2}}\bigg(\frac{\det\Delta_{\rho_{Ar}}}{\vol(M,\rho_{Ar}) \det \mathfrak{Im}(\tau)}\bigg)^{-\frac{1}{2}}\|\theta\|^{2}\big([L]-\sum_{i=1}^{N}p_{i}-\mathcal{D},\tau\big)$$  with $G_{Ar}$ the Arakelov-Green function, $(\omega_{i})_{1\leq i\leq N}$ a basis of $H^{0}(M,L)$ and $\|\det\omega_{i}(p_{j})\|_{h}^{2}=|\det\omega_{i}(p_{j})|^{2}h(p_{j}).$
Here: $$\|\theta(z,\tau)\|^{2}=\exp\left(-2\pi \mathfrak{Im}(z)^{T}(\mathfrak{Im}(\tau))^{-1}\mathfrak{Im}(z) \right)|\theta(z,\tau)|^{2}$$
\end{Proposition}

Moreover, $B_{g,k}$ are constants depending on the genus of the surface $M$ and the degree $k$ of the line bundle $L$ which have been computed explicitly in \cite{WentworthGluing} (see also \cite{GilletSoulAnalyticTorsionToddGenus}):
\begin{Proposition}[{\cite[Corollary 1.1]{WentworthGluing}}]\label{ConstantesWentworth}
    We have: $$B_{g,k}=(2\pi)^{2g-k}e^{\frac{c_{g}}{4}}$$ with $$c_{0}=-24\zeta'(-1)+1-6\ln2\pi-2\ln2, ~~  ~~ c_{1}=-8\ln2\pi ~~ \text{ and } ~~c_{g}=(1-g)c_{0}+gc_{1}.$$ \end{Proposition}

\begin{Corollary}
    In particular: $$B_{0,k}=\frac{1}{(2\pi)^{k}}\exp\bigg(-6\zeta'(-1)+\frac{1}{4}-\frac{3}{2}\ln2\pi -\frac{1}{2}\ln2\bigg),~~ \text{ and }~~ B_{1,k}=\frac{1}{(2\pi)^{k}}.$$
\end{Corollary}

\begin{Remark} We can compute for every $g$ 
     \begin{equation}\label{lnBgk}
         \ln B_{g,k}=-k\ln2\pi+\frac{1-g}{4}c_{0}=-k\ln2\pi +\frac{1-g}{2} \left(\frac{1}{2}-12\zeta'(-1)-3\ln2\pi-\ln2 \right).
         \end{equation}
 \end{Remark}
The constants $c_{g}$ in Proposition \ref{ConstantesWentworth} are the constants appearing in the Faltings' delta invariant \cite{FaltingsCalculusArithmeticSurface1984,PreciseWentwhorthI,JorgensonBoundsFaltings} 
\begin{equation}\label{FaltingDeltaEq}
\delta(M)=c_{g}-6\ln \left(\frac{\det\Delta_{M,\rho_{Ar}}}{\vol(M,\rho_{Ar})} \right).
\end{equation}
 In the bosonization formula, we see that the factor $\|\theta\|^{2}\big([L]-\sum_{i=1}^{N}p_{i}-\mathcal{D},\tau\big)$ appears, therefore, it will first be easier to compute the \emph{modified partition function}: \begin{equation}\label{ModifPartTheta}
\mathcal{Z}_{\rho,\tilde{\rho},N}^{(\theta)}(V)=\frac{1}{N!} \int_{M^{N}} \mu_{\rho}^{N}\exp\left(2 \sum_{1\leq i <j \leq N} G_{\rho}(p_{i},p_{j})+k\sum_{i=1}^{N} V(p_{i})\right)\|\theta\|^{2}\big([L]-\sum_{i=1}^{N}p_{i}-\mathcal{D},\tau\big).
\end{equation}
Indeed, from the bosonization formula, we directly see that: 
\begin{equation}\label{ModifPartTheta+Boso}
    \mathcal{Z}_{\rho_{Ar},\rho_{can},N}^{(\theta)}(V)=\frac{1}{N!}\bigg(\frac{\det\Delta_{\rho_{Ar}}}{\vol(M,\rho_{Ar}) \det \mathfrak{Im}(\tau)}\bigg)^{1/2}\int_{M^{N}}\mu_{\rho_{Ar}}^{N}\frac{\det\square_{L,\rho_{Ar},h} \|\det \omega_{i}(p_{j})\|^{2}_{h}}{B_{g,k}\det\langle \omega_{i},\omega_{j}\rangle_{\rho_{Ar},h}}e^{k\sum V(p_{i})}
\end{equation}
where $B_{g,k}$ are the explicit constants given in Proposition \ref{ConstantesWentworth}. We also know that $\det\square _{L,\rho_{Ar},h}$ admits an asymptotic expansion when $k$ tends to infinity, that can be computed using the result of \cite{FINSKI20183457}. We then have to deal with the potential $V.$ For that we will modify the Hermitian metric $h$ on the right hand side and replace it by $\hat{h}_{k,V}=e^{kV}h.$ Consequently, in the following Section we will compute the asymptotic expansion of the modified partition function defined in \eqref{ModifPartTheta}) from which we will deduce the asymptotic expansion of the partition function (see Definition \ref{DefFoncPart}) using Section \ref{AnnexeCalculTheta}. We can also note that in the case of the sphere, $g=0$, there is no theta function and consequently, in this case the modified partition function is the same as the partition function.
\section{Computation of the asymptotic expansion of the partition function}\label{SectionCalcul}
We will compute the asymptotic expansion of the modified partition function $\mathcal{Z}_{\rho_{Ar},\rho_{can},N}^{(\theta)}(V)$ defined in Equation \eqref{ModifPartTheta}. We make this choice of metrics because it simplifies the computations. The modification of the metrics are equivalent to the modification of the potential as seen in Definition \ref{DefFoncPart}. To carry out our computation, we introduce a new Hermitian metric $\hat{h}_{k,V}$ on the line bundle of degree $k$, namely \begin{equation}\label{eq:HermitianPotential} \hat{h}_{k,V}=he^{kV},
\end{equation} where $h$ is an admissible metric. Then, Equation \eqref{ModifPartTheta+Boso} becomes
\begin{equation}
    \mathcal{Z}_{\rho_{Ar},\rho_{can},N}^{(\theta)}(V)=\frac{1}{N!}\bigg(\frac{\det\Delta_{\rho_{Ar}}}{\vol(M,\rho_{Ar}) \det \mathfrak{Im}(\tau)}\bigg)^{1/2}\int_{M^{n}}\mu_{\rho_{Ar}}^{N}\frac{\det\square_{L,\rho_{Ar},h} \|\det \omega_{i}(p_{j})\|^{2}_{\hat{h}_{k,V}}}{B_{g,k}\det\langle \omega_{i},\omega_{j}\rangle_{\rho_{Ar},h}}.
\end{equation}
However, we know that: 
$$\int_{M}\mu_{\rho_{Ar}}^{N}\frac{\|\det\omega_{i}(p_{j})\|^{2}_{\hat{h}_{k,V}}}{\det\langle \omega_{i},\omega_{j}\rangle_{\rho_{Ar},\hat{h}_{k,V}}}=N!$$ and from Proposition \ref{VarDetFi}:
\begin{align*}
    \frac{\det\square_{L,\rho_{Ar},{h}}}{\det\langle \omega_{i},\omega_{j} \rangle_{\rho_{Ar},h}}&=\frac{\det\square_{L,\rho_{Ar},\hat{h}_{k,V}}}{\det\langle \omega_{i},\omega_{j}\rangle_{\rho_{Ar},\hat{h}_{k,V}}}\exp\left(-k^{2}S_{2}(-V,\rho_{Ar},h)+\frac{k}{2}S_{1}(0,-V,\rho_{Ar},h) \right),
\end{align*}
where the functionals $S_{1}$ and $S_{2}$ are defined in Definition \ref{DefFonctionnelleS}.
Combining those facts we obtain the following exact equality: \begin{equation}
    \mathcal{Z}_{\rho_{Ar},\rho_{can},N}^{(\theta)}(V)=\frac{\det\square_{L,\rho_{Ar},\hat{h}_{k,V}} }{B_{g,k}}\exp\left(-k^{2}S_{2}(-V,\rho_{Ar},h)+\frac{k}{2}S_{1}(0,-V,\rho_{Ar},h) \right)\bigg(\frac{\det\Delta_{\rho_{Ar}}}{\vol(M,\rho_{Ar}) \det \mathfrak{Im}(\tau)}\bigg)^{1/2}.
\end{equation}
In particular, for $V=0$, we have 
\begin{equation}
    \mathcal{Z}_{\rho_{Ar},\rho_{can},N}^{(\theta)}(0)=\frac{\det\square_{L,\rho_{Ar},h} }{B_{g,k}}\bigg(\frac{\det\Delta_{\rho_{Ar}}}{\vol(M,\rho_{Ar}) \det \mathfrak{Im}(\tau)}\bigg)^{1/2}
    \end{equation} where $h$ is an admissible metric, hence \begin{equation}
\frac{\mathcal{Z}_{\rho_{Ar},{\rho_{can}},N}^{(\theta)}(V)}{\mathcal{Z}_{\rho_{Ar},\rho_{can},N}^{(\theta)}(0)}=\frac{\det\square_{L,\rho_{Ar},\hat{h}_{k,V}} }{\det\square_{L,\rho_{Ar},h}}\exp\left(-k^{2}S_{2}(-V,\rho_{Ar},h)-\frac{k}{2}S_{1}(0,-V,\rho_{Ar},h) \right).
    \end{equation}

    We have then obtained the following 
    \begin{theorem}\label{TheoremPartitionGénéral}
        With $L$ a line bundle over $M$ of degree $k=N+g-1$ and $h$ an admissible metric on $L$, the following exact equality holds: 
        \begin{equation}\label{EqExactGeneral}
            \ln\mathcal{Z}_{\rho_{Ar},{\rho_{can}},N}^{(\theta)}(V)=\ln \mathcal{Z}_{\rho_{Ar},{\rho_{can}},N}^{(\theta)}(0)+\ln \det\square_{L,\rho_{Ar},\hat{h}_{k,V}}-\ln \det\square_{L,\rho_{Ar},h}+k^{2}S_{2}(V,\rho_{Ar},h)+\frac{k}{2}S_{1}(0,V,\rho_{Ar},h)
        \end{equation}
        with
        \begin{equation}\label{EqExactGeneralV=0}\ln\mathcal{Z}_{\rho_{Ar},\rho_{can},N}^{(\theta)}(0)=\ln\det\square_{L,\rho_{Ar},h} -\ln{B_{g,k}}+\frac{1}{2}\ln \bigg(\frac{\det\Delta_{\rho_{Ar}}}{\vol(M,\rho_{Ar}) \det \mathfrak{Im}(\tau)}\bigg).
        \end{equation}
    \end{theorem}

    Therefore, to obtain the desired asymptotic expansion of the partition function, we have to find the expansion of $\ln \det\square_{L,\rho_{Ar},\hat{h}_{k,V}}-\ln \det\square_{L,\rho_{Ar},h}$, which can be done using Proposition \ref{FormuleVarLapMag} and the asymptotic expansion of $\ln \det\square_{L,\rho_{Ar},h}$, which can be found in \cite{FINSKI20183457}. In addition the explicit values of $B_{g,k}$ were computed by Wentworth \cite{WentworthGluing} and are given in Proposition \ref{ConstantesWentworth} and Equation \eqref{lnBgk}.
\begin{Remark}
    We can note the interesting facts that Equations \eqref{EqExactGeneral} and \eqref{EqExactGeneralV=0} in the preceding theorem are exact equations.  
\end{Remark} 
In the preceding case, when $h$ is an admissible metric we can see that: $$S_{1}(0,-V,\rho_{Ar},h)=2(1-g) \int_{M}\mu_{\rho_{can}} V  $$ and 
\begin{equation}\label{Sinfini}
	-S_{2}(-V,\rho_{Ar},h)=\int_{M}\mu_{\rho_{can}}V\left(\frac{1}{8\pi} \Delta_{\rho_{can}}V + 1\right)=: S_{\infty}(V) 
\end{equation}
\subsection{Computation of $\ln \det\square_{L,\rho_{Ar},\hat{h}_{k,V}}-\ln \det\square_{L,\rho_{Ar},h}$}
We will compute the asymptotic expansion of $\ln \det\square_{L,\rho_{Ar},\hat{h}_{k,V}}-\ln \det\square_{L,\rho_{Ar},h}$, with $h$ an admissible metric. The result will be dependent on the potential $V.$ Proposition \ref{FormuleVarLapMag} gives:
\begin{equation}
    \ln \det\square_{L,\rho_{Ar},\hat{h}_{k,V}}-\ln \det\square_{L,\rho_{Ar},h}=\mathcal{F}\left(\rho_{Ar},h \right)-\mathcal{F}\left(\rho_{Ar},\hat{h}_{k,V} \right)+O\left(\frac{1}{k}\right),
\end{equation}
    so we have to compute $\mathcal{F}\left(\rho_{Ar},h \right)$ and $\mathcal{F}\left(\rho_{Ar},\hat{h}_{k,V} \right).$ We recall that $$\mathcal{F}(\rho,h)=\frac{1}{2\pi}\int_{M}\mu_{\rho} \left( \frac{B(\rho,h)}{2}\ln \frac{B(\rho,h)}{2\pi}+\frac{R_{\rho}}{6}\ln \frac{B(\rho,h)}{2\pi}+\frac{\ln B(\rho,h)}{24}\Delta_{\rho} \ln B(\rho,h )\right)$$ and from Definition \ref{Def:MétAdm} and Equation \eqref{ModifMag} $$B(\rho_{Ar},h)=2\pi ke^{-2\sigma_{Ar}} ~~ \text{ and } ~~ B(\rho_{Ar},\hat{h}_{k,V})=2\pi k e^{-2\sigma_{Ar}}+\frac{k}{2}\Delta_{\rho_{Ar}}V.$$ In order to use the Proposition \ref{FormuleVarLapMag}, we need the magnetic field to be positive (recall Equation \eqref{eq:CondMagPos}), hence we need to suppose that $$2\pi e^{-2\sigma_{Ar}}+\frac{1}{2}\Delta_{Ar}V>0$$ this is equivalent to: \begin{equation}\label{ConditionV}
    2\pi +\frac{1}{2}\Delta_{\rho_{can}}V>0.
\end{equation}
We will say that a function $V$ satisfying the preceding equation is \emph{quasi-subharmonic} for the metric $\rho_{can}.$
 
 We have for $h$ an admissible metric: \begin{equation}\mathcal{F}(\rho_{Ar},h)=\frac{k}{2}\ln k-k\int_{M}\mu_{\rho_{can}}\sigma_{Ar}+\frac{2}{3}(1-g)\ln k-\frac{4(1-g)}{3}\int_{M}\mu_{\rho_{can}}\sigma_{Ar}+\frac{1}{12\pi}\int_{M}\mu_{\rho_{Ar}}\sigma_{Ar}\Delta_{Ar}\sigma_{Ar}.\end{equation} and with $V$ quasi-subharmonic for the metric $\rho_{can}$ \begin{align*}
    \mathcal{F}(\rho_{Ar},\hat{h}_{k,V})&=\frac{k}{2}\ln k+k\left[\frac{1}{2}\int_{M}\mu_{\rho_{can}}\left(1+\frac{\Delta_{\rho_{can}}V}{4\pi} \right)\ln\left(1+\frac{\Delta_{\rho_{can}}V}{4\pi} \right)-\int_{M}\mu_{\rho_{can}}\sigma_{Ar}+\frac{1}{8\pi}\int_{M}\mu_{\rho_{can}}(R_{\rho_{can}}-8\pi(1-g))V \right]\\[1.5mm]&+\frac{2}{3}(1-g)\ln k+\frac{2}{3}(1-g)\int_{M}\mu_{\rho_{can}}\ln\left(1+\frac{\Delta_{\rho_{can}}V}{4\pi} \right)-\frac{4(1-g)}{3}\int_{M}\mu_{\rho_{can}}\sigma_{Ar}+\frac{1}{12\pi}\int_{M}\mu_{\rho_{Ar}}\sigma_{Ar}\Delta_{Ar}\sigma_{Ar} \\[1.5mm]&+\frac{1}{24\pi}\int_{M}\mu_{\rho_{can}}\left(R_{\rho_{can}}-8\pi(1-g) \right)\ln \left(1+\frac{\Delta_{\rho_{can}}V}{4\pi} \right)+\frac{1}{48\pi}\int_{M}\mu_{\rho_{can}}\ln \left(1+\frac{\Delta_{\rho_{can}}V}{4\pi} \right)\Delta_{\rho_{can}}\ln\left(1+ \frac{\Delta_{\rho_{can}}V}{4\pi} \right) 
\end{align*} and we obtain the following result.
\begin{theorem}\label{DifférenceLaplacienMagPotentiel}
    If $h$ is an admissible metric, $\hat{h}_{k,V}=e^{kV}h$ and $2\pi +\frac{1}{2}\Delta_{\rho_{can}}V>0$, then: 
    \begin{align*}
    \ln \det\square_{L,\rho_{Ar},\hat{h}_{k,V}}-\ln\det\square_{L,\rho_{Ar},h}&=k\left[\frac{1}{2}\int_{M}\mu_{\rho_{can}}\left(1+\frac{\Delta_{\rho_{can}}V}{4\pi} \right)\ln\left(1+\frac{\Delta_{\rho_{can}}V}{4\pi} \right)+\frac{1}{8\pi}\int_{M}\mu_{\rho_{can}}(R_{\rho_{can}}-8\pi(1-g))V \right]\\[1.5mm]&+\frac{2}{3}(1-g)\int_{M}\mu_{\rho_{can}}\ln\left(1+\frac{\Delta_{\rho_{can}}V}{4\pi} \right) +\frac{1}{24\pi}\int_{M}\mu_{\rho_{can}}\left(R_{\rho_{can}}-8\pi(1-g) \right)\ln \left(1+\frac{\Delta_{\rho_{can}}V}{4\pi} \right)\\[1.5mm]&+\frac{1}{48\pi}\int_{M}\mu_{\rho_{can}}\ln \left(1+\frac{\Delta_{\rho_{can}}V}{4\pi} \right)\Delta_{\rho_{can}}\ln\left(1+ \frac{\Delta_{\rho_{can}}V}{4\pi} \right)+O\left(\frac{1}{k} \right).
    \end{align*}
\end{theorem}

If $2\pi+\frac{1}{2}\Delta_{\rho_{can}}V>0$, we define the \emph{equilibrium volume form}, denoted $\mu_{V}$ by \begin{equation}\label{MuV}\mu_{V}:=\left(1+\frac{\Delta_{\rho_{can}}V}{4\pi} \right)\mu_{\rho_{can}}=:f_{V}\mu_{\rho_{can}}.\end{equation}

The equilibrium volume form satisfies: $$\int_{M}\mu_{V}(x)G_{\rho_{can}}(x,y)+\frac{1}{2}V(y)=2\pi\int_{M}\mu_{\rho_{can}}(x)V(x)=:c_{V}$$

and $$\int_{M}\mu_{V}=1.$$

Indeed: \begin{align*}
    \int_{M}\mu_{V}(x)G_{\rho_{can}}(x,y)&=\int_{M}\mu_{\rho_{can}}(x)\left(1+\frac{\Delta_{\rho_{can}}V(x)}{4\pi} \right)G_{\rho_{can}}(x,y) \\&=\frac{1}{4\pi}\int_{M}\mu_{\rho_{can}}(x)V(x)\Delta_{\rho_{can,x}}G_{\rho_{can}}(x,y) \\&=-\frac{V(y)}{2}+\frac{1}{2}\int_{M}\mu_{\rho_{can}}(x)V(x).
\end{align*}

With this notation, the equality in the preceding theorem can be restated as: 
    \begin{align*}
    \ln \det\square_{L,\rho_{Ar},\hat{h}_{k,V}}-\ln\det\square_{L,\rho_{Ar},h}&=k\left[\frac{1}{2}\int_{M}\mu_{V}\ln f_{V}+\frac{1}{8\pi}\int_{M}\mu_{\rho_{can}}(R_{\rho_{can}}-8\pi(1-g))V \right]\\[1.5mm]&+\frac{2}{3}(1-g)\int_{M}\mu_{\rho_{can}}\ln f_{V} +\frac{1}{24\pi}\int_{M}\mu_{\rho_{can}}\left(R_{\rho_{can}}-8\pi(1-g) \right)\ln f_{V}\\[1.5mm]&+\frac{1}{48\pi}\int_{M}\mu_{\rho_{can}}\ln (f_{V})\Delta_{\rho_{can}}\ln \left(f_{V} \right)+O\left(\frac{1}{k} \right).
    \end{align*}

\subsection{Computation of $\ln\det\square_{L,\rho_{Ar},h}$}
We consider $L\to M$ a line bundle of degree $k$ endowed with a Hermitian metric $h$ supposed to be admissible. Therefore, we have $$B(\rho_{Ar},h)=2\pi k e^{-2\sigma_{Ar}} ~~ \text{ and } ~~ B(\rho_{can},h)=2\pi k.$$ We can use the result of recalled in Theorem \ref{ThFinski} to compute $\ln\det\square_{L,\rho_{can},h}$ and then use the Proposition \ref{FormuleVarLapMag} to obtain $\ln\det\square_{L,\rho_{Ar},h}.$ We have $$\ln \det \square_{L,\rho_{can},h}=-\frac{k}{2}\ln k-\frac{k}{2}\ln2 +\frac{2(g-1)}{3}\ln k - \frac{1-g}{12}(24\zeta'(-1)+2\ln2\pi+7+8\ln2)+O\left(\frac{\ln k}{k} \right)$$
and we have $$\mathcal{F}(\rho_{can},h)=\frac{1}{2}k\ln k +\frac{2}{3}(1-g)\ln k. $$

Hence 
\begin{align*}\ln\det\square_{L,\rho_{Ar},h}&=\ln\det\square_{L,\rho_{can},h}+\mathcal{F}(\rho_{can},h)-\mathcal{F}(\rho_{Ar},h)\\&=-\frac{k}{2}\ln k +k\left(\int_{M}\mu_{\rho_{can}}\sigma_{Ar}-\frac{\ln2}{2} \right)+\frac{2(g-1)}{3}\ln k- \frac{1-g}{12}(24\zeta'(-1)+2\ln2\pi+7+8\ln2)\\&+\frac{4}{3}(1-g)\int_{M}\mu_{\rho_{can}}\sigma_{Ar}-\frac{1}{12\pi}\int_{M}\mu_{\rho_{Ar}}\sigma_{Ar}\Delta_{Ar}\sigma_{Ar} +O\left(\frac{\ln k}{k} \right).
\end{align*}
Moreover, we have from Equation \eqref{Eq:IntSigmaArDelta-PsiP} $$\int_{M} \mu_{Ar} \sigma_{Ar}\Delta_{Ar}\sigma_{Ar}=\Psi_{P}(\rho_{can})$$ and from Equation \eqref{Eq:IntSigmaAr-PsiP-SAY} $$\int_{M} \mu_{\rho_{can}}\sigma_{Ar}=2\pi(1-g)S_{AY}(\hat{\phi}_{Ar},\rho_{can})-\frac{1}{4}S_{M}(-\sigma_{Ar},-\phi_{Ar},\rho_{Ar}),$$ where $\hat{\phi}_{Ar}$ is defined in Equation \eqref{PhiHat}.
Therefore we finally obtain:
\begin{theorem}\label{DevAsympLapMagAdm}
    With $h$ an admissible metric 
    \begin{align*}
        \ln\det\square_{L,\rho_{Ar},h}&=-\frac{k}{2}\ln k +k\left(2\pi(1-g)S_{AY}(\hat{\phi}_{Ar},\rho_{can})-\frac{1}{4}S_{M}(-\sigma_{Ar},-\phi_{Ar},\rho_{Ar})-\frac{\ln2}{2}\right) +\frac{2(g-1)}{3}\ln k\\&- \frac{1-g}{12}(24\zeta'(-1)+2\ln2\pi+7+8\ln2)+\frac{\Psi_{P}(\rho_{can})-S_{L}(-\sigma_{Ar},\rho_{Ar})}{6\pi}-\frac{\Psi_{P}(\rho_{can})}{12\pi} +O\left(\frac{\ln k}{k} \right).
    \end{align*}
\end{theorem}

\subsection{Final result}
In this section we will deduce from the preceding sections the desired results about the free energy of our system.
\subsubsection{Expansion of the modified partition function}
We start by computing the asymptotic expansion of the modified partition function (see Theorems \ref{TheoremFinalK} and \ref{TheoremFinalN}) which was defined in Equation \eqref{ModifPartTheta}. We will then deduce the expansion of the partition function (see Theorem \ref{ThéorèmeFinalSansTheta}).

Combining all the previous results we obtain the following expansion for the modified partition function: 
\begin{theorem}\label{TheoremeModifPreFinal}
    Let $L \to M$ be a line bundle of degree $k$, with $k=N+1-g$, $h$ be an admissible metric on $L$, and $V$ be a smooth potential such that: $2\pi+\frac{1}{2}\Delta_{\rho_{can}}V>0.$ Then, the modified partition function satisfies:
    \begin{align*}
    \ln\mathcal{Z}_{\rho_{Ar},{\rho_{can}},N}^{(\theta)}(V)=\ln \mathcal{Z}_{\rho_{Ar},{\rho_{can}},N}^{(\theta)}(0)+\ln \det\square_{L,\rho_{Ar},\hat{h}_{k,V}}-\ln \det\square_{L,\rho_{Ar},h}+k^{2}S_{\infty}(V)+\frac{k}{2}S_{1}(0,-V,\rho_{Ar},h)
    \end{align*} where $\ln \det\square_{L,\rho_{Ar},\hat{h}_{k,V}}-\ln \det\square_{L,\rho_{Ar},h}$ is given in Theorem \ref{DifférenceLaplacienMagPotentiel}, and
    \begin{align*}
        \ln \mathcal{Z}_{\rho_{Ar},\rho_{can},N}^{(\theta)}(0)&=-\frac{k}{2}\ln k +k\left(2\pi (1-g)S_{AY}(\hat{\phi}_{Ar},\rho_{can})-\frac{1}{4}S_{M}(-\sigma_{Ar},-\phi_{Ar},\rho_{Ar})+\ln2\pi -\frac{\ln2}{2}\right)+\frac{2(g-1)}{3}\ln k \\&-(1-g)\left(-4\zeta'(-1) -\frac{4}{3}\ln2\pi +\frac{5}{6}+\frac{\ln2}{6}\right)+\frac{\Psi_{P}(\rho_ {can})-S_{L}(-\sigma_{Ar},\rho_{Ar})}{6\pi}-\frac{\Psi_{P}(\rho_{can})}{12\pi}\\&+\frac{1}{2}\ln \left(\frac{\det\Delta_{\rho_{Ar}}}{\vol(M,\rho_{Ar})\det \mathfrak{Im}(\tau)} \right)+O\left(\frac{\ln k}{k} \right).
\end{align*}
\end{theorem}

\begin{proof}
    All we need to do to obtain our result is to combine Theorems \ref{TheoremPartitionGénéral} and \ref{DevAsympLapMagAdm}. 
    We have already proved the first equality (see Theorem \ref{TheoremPartitionGénéral}). We will prove the result for $\ln \mathcal{Z}_{\rho_{Ar},\rho_{can,N}}^{(\theta)}(0).$ We also know that (see Theorem \ref{TheoremPartitionGénéral}): $$\ln \mathcal{Z}_{\rho_{Ar},\rho_{can,N}}^{(\theta)}(0)=\ln\det\square_{L,\rho_{Ar},h}-\ln B_{g,k}+\frac{1}{2}\ln \left(\frac{\det \Delta_{\rho_{Ar}}}{\vol(M,\rho_{Ar})\det\mathfrak{Im}(\tau)} \right)$$ and (see Equation\eqref{lnBgk}) $$\ln B_{g,k}=-k\ln 2\pi +\frac{1-g}{2}\left(\frac{1}{2}-12\zeta'(-1)-3\ln2\pi-\ln2 \right)$$ and (see Theorem \ref{DevAsympLapMagAdm})   \begin{align*}
        \ln\det\square_{L,\rho_{Ar},h}&=-\frac{k}{2}\ln k +k\left(2\pi(1-g)S_{AY}(\hat{\phi}_{Ar},\rho_{can})-\frac{1}{4}S_{M}(-\sigma_{Ar},-\phi_{Ar},\rho_{Ar})-\frac{\ln2}{2}\right) +\frac{2(g-1)}{3}\ln k\\&- \frac{1-g}{12}(24\zeta'(-1)+2\ln2\pi+7+8\ln2)+\frac{\Psi_{P}(\rho_{can})-S_{L}(-\sigma_{Ar},\rho_{Ar})}{6\pi}-\frac{\Psi_{P}(\rho_{can})}{12\pi} +O\left(\frac{\ln k}{k} \right).
    \end{align*}
    Hence: 
    \begin{align*}
        \ln \mathcal{Z}_{\rho_{Ar},\rho_{can},N}^{(\theta)}(0)&=-\frac{k}{2}\ln k +k\left(2\pi (1-g)S_{AY}(\hat{\phi}_{Ar},\rho_{can})-\frac{1}{4}S_{M}(-\sigma_{Ar},-\phi_{Ar},\rho_{Ar})+\ln2\pi -\frac{\ln2}{2}\right)+\frac{2(g-1)}{3}\ln k \\&-(1-g)\left(-4\zeta'(-1) -\frac{4}{3}\ln2\pi +\frac{5}{6}+\frac{\ln2}{6}\right)+\frac{\Psi_{P}(\rho_ {can})-S_{L}(-\sigma_{Ar},\rho_{Ar})}{6\pi}-\frac{\Psi_{P}(\rho_{can})}{12\pi}\\&+\frac{1}{2}\ln \left(\frac{\det\Delta_{\rho_{Ar}}}{\vol(M,\rho_{Ar})\det \mathfrak{Im}(\tau)} \right)+O\left(\frac{\ln k}{k} \right).
    \end{align*}
\end{proof}
\begin{Example}
For the case of the sphere, $g=0$, this gives (using Equation \eqref{DetLapScaArSph} and the equalities of Section \ref{Section:ArakelovCanonical}): 
\begin{align*}
    \ln \mathcal{Z}_{\rho_{Ar},\rho_{can},N}^{(\theta)}(0)&=-\frac{k}{2}\ln k+k\left(\frac{1+3\ln2\pi-2\ln2}{2} \right)-\frac{2}{3}\ln k+\frac{3}{2}\ln2\pi-\ln2+2\zeta'(-1)-\frac{1}{12}+ O\left(\frac{\ln k}{k} \right).
\end{align*}

This is consistent with direct computations in this case (see \ref{DirectSphere}): 
\begin{align*}
    \ln \mathcal{Z}^{\mathbb{S}^{2}}_{\rho_{can},\rho_{Ar},N}(0)&=-\frac{1}{2}(k+1)\ln(k+1)+(k+1)\bigg(\frac{3}{2}\ln2\pi +\frac{1}{2}-\frac{2\ln2}{2}\bigg) -\frac{\ln(k+1)}{6}+2\zeta'(-1)-\frac{1}{12}+o(1)  \\&=-\frac{1}{2}k\ln k+k\left(\frac{3}{2}\ln2\pi+\frac{1-2\ln2}{2} \right)-\frac{2}{3}\ln k-\frac{1}{12}+\frac{3}{2}\ln2\pi +2\zeta'(-1)-\ln2+o(1).
\end{align*}

Similarly, for the case of the torus, $g=1$, this gives: 
\begin{align*}
    \ln \mathcal{Z}_{\rho_{Ar},\rho_{can},N}^{(\theta)}(0)=-\frac{k}{2}\ln k+k\ln\Big(2\pi^{2} \sqrt{\mathfrak{Im}(\tau)}|\eta(\tau)|^{2}\Big)+\frac{1}{2}\ln \big(|\eta(\tau)|^{4}\big)+O\left(\frac{\ln k}{k} \right)
\end{align*}
where we used Equation \eqref{DetLapScaArTore} and Section \ref{Section:ArakelovCanonical}. We can compute directly (see \ref{AnnexeDirectTore}):
\begin{align*}
    \ln \mathcal{Z}_{\rho_{Ar},\rho_{can},N}^{(\theta)}(0)&=-\frac{1}{2}k\ln k+\ln\big(2\pi^{2}\sqrt{\mathfrak{Im}(\tau)}|\eta(\tau)|^{2} \big)k  +\frac{1}{2}\ln(|\eta(\tau)|^{4}).
\end{align*}
Consequently, the preceding formula is validated for the cases $g=0,1$.
\end{Example}
Putting everything together, we obtain the following theorem: 
\begin{theorem}\label{TheoremFinalK}
Under the same assumptions as in Theorem \ref{TheoremeModifPreFinal}, the modified partition function satisfies    \begin{equation}
        \ln \mathcal{Z}_{\rho_{Ar},\rho_{can},N}^{(\theta)}(V)=\beta_{2}k^{2}+\alpha_{1}k\ln k+\beta_{1} k+\alpha_{0}\ln k+\beta_{0}+O(k^{-1}\ln k)
    \end{equation}
    with: 
    \begin{align*}
        \beta_{2}&=S_{\infty}(V),~~~~\alpha_{1}=-\frac{1}{2},~~~~\alpha_{0}=\frac{2(g-1)}{3}=-\frac{\chi(M)}{3} \\\beta_{1}&=2\pi(1-g)S_{AY}(\hat{\phi}_{Ar},\rho_{can})-\frac{1}{4}S_{M}(-\sigma_{Ar},-\phi_{Ar},\rho_{Ar})+\ln2\pi-\frac{\ln2}{2}\\[1.2mm]&+\frac{1}{2}\left(\int_{M}\mu_{V}\ln f_{V}+\int_{M}\mu_{\rho_{can}}R_{\rho_{can}}\frac{V}{4\pi} \right)
    \end{align*}
        and 
\begin{align*}
            \beta_{0}&=(1-g)\left(4\zeta'(-1) +\frac{4}{3}\ln2\pi -\frac{5}{6}-\frac{\ln2}{6}\right)+\frac{\Psi_{P}(\rho_ {can})-S_{L}(-\sigma_{Ar},\rho_{Ar})}{6\pi}-\frac{\Psi_{P}(\rho_{can})}{12\pi}\\[1.5mm]&+\frac{2}{3}(1-g)\int_{M}\mu_{\rho_{can}}\ln f_{V} +\frac{1}{24\pi}\int_{M}\mu_{\rho_{can}}\left(R_{\rho_{can}}-8\pi(1-g) \right)\ln f_{V}\\[1.5mm]&+\frac{1}{48\pi}\int_{M}\mu_{\rho_{can}}\ln (f_{V})\Delta_{\rho_{can}}\ln \left(f_{V} \right)+\frac{1}{2}\ln \left(\frac{\det\Delta_{\rho_{Ar}}}{\vol(M,\rho_{Ar})\det \mathfrak{Im}(\tau)} \right)
        \end{align*}
where $\chi(M)$ is the Euler characteristic of the Riemann surface $M$, $S_{2}$ and $S_{1}$ are defined in Definition \ref{DefFonctionnelleS}, $S_{\infty}$ is defined in Equation \eqref{Sinfini}, $S_{AY}$, $S_{M}$ and $S_{L}$ are defined in Definition \ref{DefLiouvMabAubYau}, $\Psi_{P}$ is defined in Definition \ref{PotRicci} and $\mu_{V}$ and $f_{V}$ are defined in Equation \eqref{MuV}. 
\end{theorem}

Now using the fact that $k=N+g-1$, we can see easily that we have the following asymptotic expansion in $N$.
\begin{theorem}\label{TheoremFinalN}
    With the same conditions as before: 
    \begin{equation}
\ln \mathcal{Z}_{\rho_{Ar},\rho_{can},N}^{(\theta)}(V)=b_{2}N^{2}+a_{1}N\ln N+b_{1} N+a_{0}\ln N+b_{0}+O(N^{-1}\ln N)
    \end{equation}
    with: \begin{align*}
        b_{2}&=\beta_{2}, ~~~~a_{1}=\alpha_{1}, ~~~~b_{1}=2(g-1)\beta_{2}+\beta_{1},~~~~ a_{0}=\alpha_{0}+(g-1)\alpha_{1} \\b_{0}&=(g-1)^{2}\beta_{2}+(g-1)\alpha_{1}+(g-1)\beta_{1}+\beta_{0}
    \end{align*}
    where the $\alpha_{i}$ and $\beta_{i}$ are given in Theorem \ref{TheoremFinalK} above.
\end{theorem}
In particular, for the simpler case $V=0$ where the potential is null (which clearly satisfies the condition of Equation \eqref{ConditionV}): 
\begin{Corollary}
The modified partition function with potential $V=0$ satisfies: 
    \begin{align*}
\ln \mathcal{Z}_{\rho_{Ar},\rho_{can},N}^{(\theta)}(0)&=-\frac{N}{2}\ln N +N\left(2\pi(1-g)S_{AY}(\hat{\phi}_{Ar},\rho_{can})-\frac{1}{4}S_{M}(-\sigma_{Ar},-\phi_{Ar},\rho_{Ar})+\ln2\pi-\frac{\ln2}{2}\right) +\frac{(g-1)}{6}\ln N\\&+ (1-g)\left(4\zeta'(-1)+\frac{1}{3}\ln2\pi-\frac{1}{3}+\frac{1}{3}\ln2\right)-\frac{\Psi_{P}(\rho_{can})}{24\pi} +\frac{1}{2}\ln\left(\frac{\det\Delta_{\rho_{can}}}{\det \mathfrak{Im}(\tau)}\right) +O\left(\frac{\ln N}{N} \right)
    \end{align*}
\end{Corollary}

The constant term can also be rewritten using Faltings' delta invariant (see Equation \eqref{FaltingDeltaEq}). 
Indeed, 
$$\frac{1}{2}\ln \left(\frac{\det\Delta_{\rho_{Ar}}}{\vol(M,\rho_{Ar})} \right)=(1-g)\left(\frac{1}{12}-2\zeta'(-1)-\frac{1}{2}\ln2\pi -\frac{1}{6}\ln2 \right)-\frac{2}{3}g\ln2\pi-\frac{1}{12}\delta(M).$$
Hence 
\begin{align*}
    \ln \mathcal{Z}_{\rho_{Ar},\rho_{can},N}^{(\theta)}(0)&=-\frac{N}{2}\ln N +N\left(2\pi(1-g)S_{AY}(\hat{\phi}_{Ar},\rho_{can})-\frac{1}{4}S_{M}(-\sigma_{Ar},-\phi_{Ar},\rho_{Ar})+\ln2\pi\right) +\frac{(g-1)}{6}\ln N-\frac{2}{3}g\ln2\pi\\&+ (1-g)\left(2\zeta'(-1)-\frac{1}{6}\ln2\pi-\frac{1}{4}+\frac{1}{6}\ln2\right)-\frac{\Psi_{P}(\rho_{can})-S_{L}(-\sigma_{Ar},\rho_{Ar})}{24\pi} -\frac{1}{2}\ln\det \mathfrak{Im}(\tau)-\frac{\delta(M)}{12} +O\left(\frac{\ln N}{N} \right).
\end{align*}

\begin{Remark}
    For $g=0$, The Faltings' delta invatiant is $\delta(\mathbb{S}^{2})=0$ (see Equation \eqref{Eq:FaltingSphere}) and for $g=1$, we know that  $\delta(\mathbb{T^{2}_{\tau}})=-6\ln \big( \mathfrak{Im}(\tau)|\eta(\tau)|^{4} \big)-8\ln2\pi$ (see Equation \eqref{Eq:FaltingsTore}).
\end{Remark}
We thus have found an explicit asymptotic expansion for the logarithm of the modified partition function. In the following, we will use it to deduce the asymptotic expansion of the logarithm of the partition function.

\begin{Remark}
    For $g=0$, we have $\mathcal{Z}^{(\theta)}=\mathcal{Z},$ so for this particular case, the preceding result can be considered the final result.
\end{Remark}
\subsubsection{Integral of theta function}\label{AnnexeCalculTheta}
To find the expansion of the partition function from the expansion of the modified partition function, we need to eliminate the theta function. To do so, we consider $M$ of genus $g\geq1$ and we will do an integration of the theta function on the Jacobian torus. Consequently, we need to compute  
$$\int_{\mathbb{T}^{2g}} \|\theta(z,\tau)\|^{2} dZ$$ for $$\|\theta(z,\tau)\|^{2}=\exp\left(-2\pi \mathfrak{Im}(z)\mathfrak{Im}(\tau)^{-1}\mathfrak{Im}(z) \right)|\theta(z,\tau)|^{2}$$ and $$dZ=\frac{dz}{\det\mathfrak{Im}(\tau)},$$ where we recall that $\mathfrak{Im}(\tau)$ is a positive definite matrix of size $g.$ 
This computation is used to obtain the difference between the modified partition function and the partition function (see Section \ref{SectionFinal} below). This computation is also useful for Appendix \ref{AnnexeDirectTore}.
We will now prove the following useful Lemma.
\begin{Lemma}
    The following equality holds: $$\int_{\mathbb{T}^{2g}_{\tau}} \|\theta(Z,\tau)\|^{2} dZ=\left(\frac{1}{2^{g}\det\mathfrak{Im}(\tau)}\right)^{1/2}$$
\end{Lemma} 
\begin{proof}We will note for $z \in \mathbb{T}^{2g}, $ $z=P+\tau Q$, with $P$ and $Q$ two real vectors. With this notation: $$\|\theta(z,\tau)\|^{2}=\exp\left(-2\pi Q^{T}\mathfrak{Im}(\tau) Q \right)|\theta(z,\tau)|^{2}$$
and $$\theta(z,\tau)=\sum_{n \in \mathbb{Z}^{g}}\exp\left(i\pi n^{T}\tau n+2i\pi n^{T}z \right)$$
hence $$\bar{\theta}(z,\tau)=\sum_{n \in \mathbb{Z}^{g}}\exp\left(-i\pi n^{T}\bar{\tau} n-2i\pi n^{T}\bar{z} \right)$$
and \begin{align*}|\theta(z,\tau)|^{2}&=\sum_{n\in \mathbb{Z}^{g}}\sum_{\ell\in \mathbb{Z}^{g}}\exp\left(i\pi n^{T}\tau n+2i\pi n^{T}z -i\pi \ell^{T}\bar{\tau} l-2i\pi \ell^{T}\bar{z}\right)\\&=\sum_{n\in \mathbb{Z}^{g}}\sum_{\ell\in \mathbb{Z}^{g}}\exp\left(i\pi n^{T}\tau n -i\pi \ell^{T}\bar{\tau} \ell+2i\pi (n^{T}\tau - \ell^{T}\bar{\tau})Q\right)\exp\left(2i\pi (n^{T}-\ell^{T})P \right).\end{align*}

We want to compute: $$\int_{\mathbb{T}^{2g}}\|\theta(z,\tau)\|^{2}dP dQ. $$ From Fubini's theorem, we obtain: 
\begin{align*}
\int_{\mathbb{T}^{2g}}\|\theta(z,\tau)\|^{2}dP dQ&=\sum_{n\in \mathbb{Z}^{g}}\int_{[0,1]^{g}}\exp\left(i\pi n^{T}(\tau-\bar{\tau}) n +2i\pi n^{T}(\tau - \bar{\tau})Q\right)\exp\left(-2\pi Q^{T}\mathfrak{Im}(\tau) Q \right)dQ \\&=\sum_{n\in \mathbb{Z}^{g}}\int_{[0,1]^{g}}\exp \left(-2\pi n^{T}\mathfrak{Im}(\tau)n-4\pi n^{T}\mathfrak{Im}(\tau)Q-2\pi Q^{T}\mathfrak{Im}(\tau)Q \right)dQ \\&=\sum_{n\in \mathbb{Z}^{g}}\int_{[0,1]^{g}}\exp\left(-2\pi (n+Q)^{T}\mathfrak{Im}(\tau)(n+Q)\right) dQ \\&=\sum_{n\in \mathbb{Z}^{g}}\int_{n+[0,1]^{g}}\exp(-2\pi V^{T}\mathfrak{Im}(\tau) V)dV \\&=\int_{\mathbb{R}^{g}}\exp\left(-\frac{1}{2}U^{T}\mathfrak{Im(\tau)}U \right)\frac{dU}{(4\pi)^{g/2}}\\&=\left(\frac{1}{2^{g}\det\mathfrak{Im}(\tau)}\right)^{1/2}
\end{align*}
because $\mathfrak{Im}(\tau)$ is positive definite.
\end{proof}

\begin{Remark}
    The preceding computation makes sense only for $g\geq1$, but for the case $g=0$, the modified partition function is equal to the partition function, so this procedure is unnecessary in this special case.
\end{Remark}

\subsubsection{Expansion of the partition function}\label{SectionFinal}
In Theorem \ref{TheoremFinalN}, we computed the asymptotic expansion of the free energy for the modified partition function, we now want to obtain the asymptotic expansion of the free energy for the partition function defined in Definition \ref{DefFoncPart}. 

Using the computations of Section \ref{AnnexeCalculTheta}, we can deduce from Theorem \ref{TheoremFinalN} the expansion of the partition function. Indeed, in the left hand side of the expansion of Theorem \ref{TheoremFinalN}, we can integrate the $[L]$ appearing in the bosonization formula and the terms in right hand side do not depend of this variable anymore, moreover the rest $O\left(N^{-1}\ln N \right)$ coming from Proposition \ref{ThFinski} and Proposition \ref{VarDetFi} can be bounded uniformly in the Jacobian (see Remark \ref{UniformeFinski}). Therefore, since $$\mathcal{Z}_{\rho_{Ar},\rho_{can},N}(V)=(2^{g}\det \mathfrak{Im}(\tau))^{1/2}\int_{\Jac(M)}\mathcal{Z}^{(\theta)}_{\rho_{Ar},\rho_{can},N}(V)$$ we find: $$\ln \mathcal{Z}_{\rho_{Ar},\rho_{can},N}(V)=\ln \mathcal{Z}_{\rho_{Ar},\rho_{can,N}}^{(\theta)}(V)+\frac{1}{2}\ln\det\mathfrak{Im}(\tau)+\frac{g}{2}\ln2+O\left(\frac{\ln N}{N} \right).$$

Indeed, our computation consists in varying the line bundle $E$ in $L=L_{0}^{k}\otimes E$ and keeping the curvature  of the line bundle $L$ (namely the magnetic field) constant, consequently from Remark \ref{UniformeFinski} and Theorem \ref{TheoremFinalN}: 

\begin{align*}
    \int_{\Jac(M)} \mathcal{Z}^{(\theta)}_{\rho_{Ar},\rho_{can},N}(V)&=\exp\left(b_{2}N^{2}+a_{1} N\ln N +a_{0} \ln N +b_{0}\right)\int_{M}\exp \left(O(N^{-1}\ln N) \right) \\&=\exp \left(b_{2}N^{2}+a_{1} N\ln N +a_{0} \ln N +b_{0}+O\left(N^{-1}\ln N \right)\right)
\end{align*}
by the aforementioned uniformity.
Thus, we obtain our final Theorem: 
\begin{theorem}\label{ThéorèmeFinalSansTheta}
    We have for any potential $V$ such that $2\pi+\frac{1}{2}\Delta_{\rho_{can}}V>0$: 
    \begin{equation}
        \ln \mathcal{Z}_{\rho_{Ar},\rho_{can},N}(V)=\ln \mathcal{Z}_{\rho_{Ar},\rho_{can},N}(0)+\ln\det\square_{L,\rho_{Ar},\hat{h}_{k,V}}-\ln\det\square_{L,\rho_{Ar},h}+k^{2}S_{\infty}(V)+\frac{k}{2\pi}(1-g)c_{V}+O(N^{-1}\ln N)
    \end{equation}
    where $$\ln \mathcal{Z}_{\rho_{Ar},\rho_{can},N}(0)=\ln\det\square_{L,\rho_{Ar},h}-\ln B_{g,k}+\frac{1}{2}\ln  \left(\frac{\det\Delta_{\rho_{Ar}}}{\vol(M,\rho_{Ar})}\right)+\frac{g}{2}\ln2+O(N^{-1}\ln N).$$
    This gives:
    \begin{equation}
        \ln \mathcal{Z}_{\rho_{Ar},\rho_{can},N}(V)=B_{2}N^{2}+A_{1}N\ln N+B_{1} N+A_{0}\ln N+B_{0}+O(N^{-1}\ln N)
    \end{equation}
    with: \begin{align*}
                B_{2}=b_{2},~~~~ A_{1}=a_{1}, ~~~~ B_{1}=b_{1}, ~~~~ A_{0}=a_{0},~~~~ B_{0}=b_{0}+\frac{1}{2}\ln\det\mathfrak{Im}(\tau)+\frac{g}{2}\ln2
    \end{align*}
    with the $b_{i}$ and $a_{i}$ as in Theorem \ref{TheoremFinalN}.

Explicitly, we have: 
\begin{align*}
    B_{2}&=S_{\infty}(V),~~~~  A_{1}=-\frac{1}{2}, ~~~~ A_{0}=\frac{g-1}{6}=-\frac{\chi(M)}{12} \\B_{1}&=2(g-1)S_{\infty}(V)+2\pi(1-g)S_{AY}(\hat{\phi}_{Ar},\rho_{can})-\frac{1}{4}S_{M}(-\sigma_{Ar},-\phi_{Ar},\rho_{Ar})+\ln2\pi-\frac{\ln2}{2}\\&+\frac{1}{2}\left(\int_{M}\mu_{V}\ln\left(f_{V} \right)+\int_{M}\mu_{\rho_{can}}R_{\rho_{can}}\frac{V}{4\pi} \right) 
    \end{align*}
    and
    \begin{align*}
        B_{0}&=(g-1)^{2}S_{\infty}(V)+ (1-g)\left(4\zeta'(-1)+\frac{1}{3}\ln2\pi-\frac{1}{3}+\frac{1}{3}\ln2\right)-\frac{\Psi_{P}(\rho_{can})}{24\pi} \\[1.7mm]&+\frac{(g-1)}{2}\left(\int_{M}\mu_{V}\ln\ f_{V}+\int_{M}\mu_{\rho_{can}}R_{\rho_{can}}\frac{V}{4\pi} \right) \\[1.7mm]&+\frac{2}{3}(1-g)\int_{M}\mu_{\rho_{can}}\ln f_{V} +\frac{1}{24\pi}\int_{M}\mu_{\rho_{can}}\left(R_{\rho_{can}}-8\pi(1-g) \right)\ln f_{V}\\[1.7mm]&+\frac{1}{48\pi}\int_{M}\mu_{\rho_{can}}\ln (f_{V})\Delta_{\rho_{can}}\ln \left(f_{V} \right)+\frac{1}{2}\ln \left(\det\Delta_{\rho_{can}}\right)+\frac{g}{2}\ln2
    \end{align*}
with $\chi(M)=2(1-g)$ the Euler characteristic of $M$, $S_{\infty}$ a functional defined in Equation \eqref{Sinfini}, $\Psi_{P}$ the Polyakov functional defined in Definition \ref{PotRicci}. $S_{AY},$ $S_{M}$, and $S_{L}$ are respectively the Aubin-Yau functional, the Mabuchi functional and the Liouville functional which are all defined in Definition \ref{DefLiouvMabAubYau}, $\mu_{V}$ is the equilibrium volume form defined together with $f_{V}$ in Equation \eqref{MuV}, and $\hat{\phi}_{Ar}$ was defined in Equation \eqref{PhiHat}.
\end{theorem}
\begin{Remark}
    The preceding theorem is true in any genus $g.$ Moreover, we gave in the previous theorem an expansion in $N$, but we should note that with the same method we could obtain an asymptotic expansion in $k$, which would have a simpler expression.
\end{Remark}

The preceding result gives an explicit asymptotic expansion of the free energy in the determinantal case and we find that the determinant of the scalar Laplacian appears in the constant term of the expansion thus confirming the geometric Zabrodin-Wiegmann conjecture.
\begin{Remark}
    This result has to be compared with the one conjectured for the Coulomb gaz in the plane $\mathbb{C}$ for general $\beta$ \cite[§ 9.3]{serfaty2024lecturescoulombrieszgases}
\begin{align*}
    \ln \left(\frac{1}{N!}Z_{N,\beta}\right)&=-2\beta N^{2} \mathcal{E}(\mu_{V})+\left(\frac{\beta}{4}-\frac{1}{2}\right)N\ln N+\left(2\beta f_{d}(\beta)+1+\left(1-\frac{\beta}{2} \right)\int_{M}\mu_{V} \ln \mu_{V} \right)N+\frac{4}{3\sqrt{\pi}}\ln \left(\beta \right)\sqrt{N}\\&-\frac{\chi}{12}\ln N+E_{\beta}-\frac{1}{2}\ln2\pi+o(1)
\end{align*}
where our computation was done for the determinantal case $\beta=1$ on compact Riemann surfaces and in the case we considered, the physical potential was $-\frac{V}{2}$. 
Note also that we chose different notations than \cite{serfaty2024lecturescoulombrieszgases}, the determinantal case $\beta=1$ corresponds to the case $\beta=2$ in \cite{serfaty2024lecturescoulombrieszgases}.
\end{Remark}
In pariticular, we have without potential:

\begin{Corollary}
In the simpler case with no potential, $V=0$, the expansion of the partition function is:
    \begin{align*}
    \ln \mathcal{Z}_{\rho_{Ar},\rho_{can},N}(0)&=-\frac{N}{2}\ln N +N\left(2\pi(1-g)S_{AY}(\hat{\phi}_{Ar},\rho_{can})-\frac{1}{4}S_{M}(-\sigma_{Ar},-\phi_{Ar},\rho_{Ar})+\ln2\pi-\frac{\ln2}{2}\right) +\frac{(g-1)}{6}\ln N\\&+ (1-g)\left(4\zeta'(-1)+\frac{1}{3}\ln2\pi-\frac{1}{3}+\frac{1}{3}\ln2\right)-\frac{\Psi_{P}(\rho_{can})}{24\pi} +\frac{1}{2}\ln\left(\det\Delta_{\rho_{can}}\right)+\frac{g}{2}\ln2 +O\left(\frac{\ln N}{N} \right).
    \end{align*}
\end{Corollary}
\begin{Example}
For $g=0$, using Equation \eqref{DetLapScaArSph} the previous theorem gives: 
$$\ln\mathcal{Z}_{\rho_{Ar},\rho_{can,N}}(0)=-\frac{N}{2}\ln N + N\left(\frac{1+3\ln2\pi}{2}-\ln2 \right)-\frac{1}{6}\ln N+2\zeta'(-1)-\frac{1}{12}+O\left(\frac{\ln N}{N} \right)$$
and for $g=1$, using Equation \eqref{DetLapScaArTore}: 
$$\ln\mathcal{Z}_{\rho_{Ar},\rho_{can,N}}(0)=-\frac{N}{2}\ln N+N\ln\Big(2\pi^{2} \sqrt{\mathfrak{Im}(\tau)}|\eta(\tau)|^{2}\Big)+\frac{1}{2}\ln \big(2\mathfrak{Im}(\tau)|\eta(\tau)|^{4}\big)+O\left(\frac{\ln N}{N} \right).$$
\end{Example}

\begin{Remark}
    All the preceding computations were made with the arbitrary choice $$\int_{M}\mu_{\rho_{can}}G_{\rho_{can}}=0.$$ If, instead we consider $$\tilde{G}_{\lambda,\rho_{can}}(x,y)=G_{\rho_{can}}(x,y)+\lambda$$ with $\lambda$ a constant. We can see that: $$\int_{M^{2}}\mu_{\rho_{can}}^{2}(x,y)\tilde{G}_{\lambda,\rho_{can}}(x,y)=\lambda.$$ If we define $$\tilde{\mathcal{Z}}_{\lambda,\rho_{Ar},\rho_{can}}(V)=\frac{1}{N!}\int_{M^{N}}\mu_{\rho_{Ar}}^{N}\exp\left(2\sum_{1\leq i < j \leq N} \tilde{G}_{\lambda,\rho_{can}}(p_{i},p_{j}) +k\sum_{i=1}^{N} V(p_{i})\right),$$ then we directly see from the definition that, if $2\pi+\frac{1}{2}\Delta_{\rho_{can}}V>0,$ $$        \ln \tilde{\mathcal{Z}}_{\lambda,\rho_{Ar},\rho_{can},N}(V)=\tilde{B}_{2}N^{2}+A_{1}N\ln N+\tilde{B}_{1} N+A_{0}\ln N+B_{0}+O(N^{-1}\ln N),
$$ with $$\tilde{B}_{2}=B_{2}+\lambda, ~~ \text{ and } ~~ \tilde{B}_{1}=B_{1}-\lambda$$
and the $A_{i}$ and $B_{i}$ are given in Theorem \ref{ThéorèmeFinalSansTheta}.
\end{Remark}
\subsection{Fluctuations}
In this section, we will consider the convergence of the fluctuations, in the same setting as before, proving Corollary \ref{FluctuationsConvergence} using Theorem \ref{ThéorèmeFinalSansTheta} and following the method of \cite{ameur2025freeenergyfluctuationsrandom,ameur2025twodimensionalcoulombgasfluctuations}. 
Associated with the determinantal Coulomb gas with $V$ quasi-subharmonic is the probability measure:
\begin{equation}\label{eq:Proba} \mathrm{d}\mathbb{P}_{\rho_{can},\rho_{Ar},N}^{(V)}=\frac{N!}{\mathcal{Z}_{\rho_{can},\rho_{Ar},N}(V)}\exp\left(2\sum_{1\leq i < j \leq N} G_{\rho_{can}}(p_{i},p_{j}) +k\sum_{i=1}^{N} V(p_{i})\right)\mu_{\rho_{Ar}}^{N}
\end{equation}

We consider $f$ to be a non-constant smooth function on $M$ and we definr $W_{s}=V+\frac{sf}{k}$ where $k=N+g-1$ and $V$ is quasi-subharmonic. Then, for $0< s <\frac{4\pi k \min|f_{V}|}{\max|\Delta_{\rho_{can}} f| }$, with $f_{V}$ defined in Equation \eqref{MuV}, $W_{s}$ is also quasi-subharmonic (see Equation \eqref{ConditionV}). We will define: $$\Fluct_{N}f=\sum_{i=1}^{N} f(z_{i})-N\int_{M} \mu_{V} f. $$ 

We note $F_{N,f}(s)=\log \mathbb{E}_{N}^{(V)} \left(e^{s \Fluct_{N}f} \right) $ the cumulant generating functional of $\Fluct_{N}f$

Then, from \cite[Section 1.3.1]{ameur2025freeenergyfluctuationsrandom} and \cite[Section 1.3.2]{ameur2025twodimensionalcoulombgasfluctuations}, we can find: $$ \frac{d}{ds} F_{N,f}(s)=\frac{d}{ds} \ln \mathcal{Z}_{\rho_{Ar},\rho_{can}}(W_{s})-N\int_{M} \mu_{V}f $$ and from Theorem \ref{ThéorèmeFinalSansTheta}, we have when $k=N+g-1\to \infty$: $$\frac{d}{ds}\ln \mathcal{Z}_{\rho_{Ar},\rho_{can}}(W_{s})=k\int_{M} \mu_{V} f + \frac{s}{4\pi}\int_{M} \mu_{\rho_{can}} f \Delta_{\rho_{can}} f +\frac{1}{8\pi}\int_{M}\mu_{\rho_{can}}R_{\rho_{can}}f+\frac{1}{8\pi}\int_{M} \mu_{\rho_{can}} f\Delta_{\rho_{can}}\ln(f_{V})+o(1) $$ with $R_{\rho_{can}}$ the scalar curvature and where $\mu_{V}$ and $f_{V}$ are defined in Equation \eqref{MuV}.

Then integrating and using the fact that $F_{N,f}(0)=0$, we find $$F_{N,f}(s) = \frac{s^{2}}{2}v_{f}+sm_{f}+o(1), $$ with $v_{f}$ and $m_{f}$ as in Corollary \ref{FluctuationsConvergence}. We recognize the cumulant generating function of a random normal variable with mean $m_{f}$ and variance $v_{f}$, thus proving Corollary \ref{FluctuationsConvergence}.
\appendix
\section{Exact computations}
We will now present another computation of the same quantities and following the same method using explicit computations of the determinant of the magnetic Laplacian when they are available. In particular, those computations will show that for the torus, $g=1$, we have in fact an exact formula and not only an asymptotic expansion. First, we will check our main result for the sphere ($g=0$) using explicit formulas for the determinant of the magnetic Laplacian. Then, we will do a direct computation of the partition function for the sphere using the explicit Green function and Arakelov metric (see Section \ref{SectionArakelovMetric}). We will then do the same thing for the genus $g=1$. Finally, we will again check our result for $g>1$ using the explicit formula of the determinant of the magnetic Laplacian when $k=2n(g-1)$ with $n \in \mathbb{N}$ and taking $L$ to be the canonical bundle.
\subsection{Case of genus $g=0$}
For the genus $g=0$, we already know from Proposition \ref{Prop:VolumeArakelov} that $\vol(\mathbb{S}^{2},\rho_{Ar})=\pi e.$ Moreover, the determinant of the scalar Laplacian can be explicitly computed and for the round metric $\rho_{0}$ of volume $4\pi$, we know that (\cite[eq. (4.77)]{ExplicitSphere},\cite[Th. 4.1]{ChoiDetLapS3-1994}) $$\det\Delta_{\rho_{0}}=\exp\left(\frac{1}{2}-4\zeta'(-1) \right).$$ Hence, using Equation (\ref{VarDetLapScaCste}) we obtain: 
\begin{equation}\label{DetLapScaArSph}\det\Delta_{\rho_{Ar}}=\exp\left(\frac{7}{6}-\frac{4}{3}\ln2 -4\zeta'(-1)\right).
\end{equation}
We also deduce from what precedes: \begin{equation*}
    \left(\frac{\det\Delta_{\rho_{Ar}}}{\vol(\mathbb{S}^{2},\rho_{Ar})} \right)^{1/2}=\exp\left(\frac{1}{12}-\frac{1}{6}\ln2-2\zeta'(-1)-\frac{1}{2}\ln2\pi \right).
\end{equation*}
Similarly, for the sphere with the Fubini-Study metric $\rho_{FS}$ of volume $2\pi$, we know that, for $L=\mathcal{O}(k)$, and usual admissible metric locally given by $h(z)=\frac{1}{(1+|z|^{2})^{k}}$ we have (\cite[eq. (74)]{KlevtsovQuillen},\cite[p. 352]{WengDeterminantLaplacian}): 
$$\det\square_{L,\rho_{FS},h}=\exp\bigg((k+1)\ln(k+1)!-2\sum_{j=1}^{k+1}j\ln j-4\zeta'(-1)+\frac{(k+1)^{2}}{2}-\frac{k}{2}\ln2-\frac{2}{3}\ln2\bigg)$$ where the terms $-\frac{k}{2}\ln2$ and $-\frac{2}{3}\ln2$ come from the fact that we consider $\square_{L,\rho_{FS},h}=2\bar{\partial}^{\dagger}_{L}\bar{\partial}_{L}$ instead of $\bar{\partial}^{\dagger}_{L}\bar{\partial}_{L}$.  Hence, for the Arakelov metric of volume $\pi e$: \begin{equation}\label{DetLapMagArSph}
    \det \square_{L,\rho_{Ar},h}=\exp\bigg((k+1)\ln(k+1)!-2\sum_{j=1}^{k+1}j\ln j-4\zeta'(-1)+\frac{(k+1)^{2}}{2}+\frac{1-2\ln2}{2}(k+1)+\frac{1-2\ln2}{6}\bigg).
\end{equation}
Then, from the Bosonization formula applied to the sphere, we have: \begin{equation}
    \mathcal{Z}_{\rho_{Ar},\rho_{can},N}(0)=\left(\frac{\det \Delta_{\rho_{Ar}}}{\vol(\mathbb{S}^{2},\rho_{Ar})}\right)^{1/2}\frac{\det\square_{L,\rho_{Ar},h}}{B_{0,k}}
\end{equation}
and we know the explicit expression of every term in the right hand side.
Thus, \begin{equation}
    \ln \mathcal{Z}_{\rho_{Ar},\rho_{can},N}(0)=\frac{1}{2}\ln\left(\frac{\det \Delta_{\rho_{Ar}}}{\vol(\mathbb{S}^{2},\rho_{Ar})}\right)+\ln\det\square_{L,\rho_{Ar},h}-\ln B_{0,k}.
\end{equation}
We know that $\ln B_{0,k}=-k\ln2\pi -6\zeta'(-1)+\frac{1}{4}-\frac{3}{2}\ln2\pi-\frac{1}{2}\ln2.$
And using the Barnes G-function (see Equation \eqref{Eq:Barnes}), the Stirling formula and the asymptotic expansion of $\sum j\ln j$ (see Equation \eqref{Eq:jlnj}), we can compute 
\begin{align*}
    \ln \det\square_{L,\rho_{Ar},h}=-\frac{k}{2}\ln k +\left(\frac{1+\ln 2\pi-\ln2}{2} \right)k-\frac{2}{3}\ln k+\frac{1}{12}-2\zeta'(-1)+\frac{1}{2}\ln2\pi -\frac{2}{3}\ln2+o(1).
\end{align*}
Therefore: \begin{equation}
    \ln \mathcal{Z}_{\rho_{Ar},\rho_{can},N}(0)=-\frac{k}{2}\ln k+\left(\frac{1+3\ln2\pi-2\ln2}{2} \right)k-\frac{2}{3}\ln k+2\zeta'(-1)-\frac{1}{12}-\ln2+\frac{3}{2}\ln2\pi+o(1).
\end{equation}
We can also note, using Equation \eqref{DetLapScaArSph} that Faltings' delta invariant (see Equation \eqref{FaltingDeltaEq}) satisfies: \begin{equation}\label{Eq:FaltingSphere}\delta(\mathbb{S}^{2})=0. \end{equation}

\subsubsection{Direct computation for the sphere}\label{DirectSphere}
For the sphere, we will compute $\ln\mathcal{Z}_{\rho_{Ar},\rho_{can},N}(0)$ with a direct method. We know that $$\exp G_{\rho_{can}}(z,w)=\frac{e^{\frac{1}{2}}|z-w|}{\sqrt{(1+|z|^{2})(1+|w|^{2})}}$$ and $$\mu_{\rho_{Ar}}=\frac{ie}{2(1+|z|^{2})^{2}}dz\wedge d\bar{z}.$$ Hence, we have: 
\begin{align*}
    \mathcal{Z}_{\rho_{Ar},\rho_{can},N}(0)&=\frac{1}{N!}\int_{(\mathbb{S}^{2})^{N}}\mu_{\rho_{Ar}}^{N}\prod_{1\leq i<j\leq N}\exp \Big(G_{\rho_{can}}(z_{i},z_{j})\Big)^{2} \\&=\frac{1}{N!}\int_{(\mathbb{S}^{2})^{N}}\prod_{1\leq j\leq N}\frac{ie}{2(1+|z_{j}|^{2})^{2}}dz_{j}\wedge d\bar{z}_{j} \prod_{1\leq i<j \leq N}\frac{e|z_{i}-z_{j}|^{2}}{(1+|z_{i}|^{2})(1+|z_{j}|^{2})}\\&=\frac{e^{\frac{N(N+1)}{2}}}{2^{N}N!}\int_{(\mathbb{S}^{2})^{N}}\left(\prod_{1\leq i<j \leq N}|z_{i}-z_{j}| \right)^{2}\prod_{1\leq j \leq N} \frac{idz_{j}\wedge d\bar{z}_{j}}{(1+|z_{j}|^{2})^{N+1}} \\&=\frac{e^{\frac{N(N+1)}{2}}}{2^{N}N!}\int_{(\mathbb{S}^{2})^{N}}\left|\det (f_{\ell}(z_{i},1)) \right|^{2}\prod_{1\leq j \leq N} \frac{idz_{j}\wedge d\bar{z}_{j}}{(1+|z_{j}|^{2})^{N+1}}
\end{align*}
where, for $1\leq \ell\leq N$, $$f_{\ell}(z_{i},z_{j})=z_{i}^{\ell-1}z_{j}^{N-\ell}.$$
Therefore: \begin{align*}
    \mathcal{Z}_{\rho_{Ar},\rho_{can},N}(0)&=\frac{e^{\frac{N(N+1)}{2}}}{2^{N}}\det\left(\int_{\mathbb{S}^{2}(M)}f_{j}(z,1)\bar{f}_{\ell}(z,1)\frac{id z\wedge d\bar{z}}{(1+|z|^{2})^{N+1}} \right)_{j,\ell}.
\end{align*}

\begin{Lemma}
    We can see that: 
    \begin{equation*}
       \forall j,i, \quad\quad \int_{\mathbb{S}^{2}(M)}f_{j}(z,1)\bar{f}_{\ell}(z,1)\frac{id z\wedge d\bar{z}}{(1+|z|^{2})^{N+1}} =\frac{2\pi}{N \binom{N-1}{\ell-1}}\delta_{j,\ell}.
    \end{equation*}
\end{Lemma}
\begin{proof}
    \begin{align*}
        \int_{\mathbb{S}^{2}(M)}f_{j}(z,1)\bar{f}_{\ell}(z,1)\frac{id z\wedge d\bar{z}}{(1+|z|^{2})^{N+1}}&=\int_{\mathbb{S}^{2}(M)}\frac{z^{j-1}\bar{z}^{\ell-1}}{(1+|z|^{2})^{N+1}}idz\wedge d \bar{z} \\&=\int_{0}^{2\pi} \exp(i\theta(j-\ell))d\theta\int_{r=0}^{\infty}\frac{2r^{j+\ell-1}}{(1+|r|^{2})^{N+1}}.
    \end{align*}
    If $j\neq \ell$, we see that the integral with respect to $\theta$ is zero. For $j=\ell$, we have: 
    \begin{align*}
        \int_{\mathbb{S}^{2}(M)}f_{j}(z,1)\bar{f}_{\ell}(z,1)\frac{id z\wedge d\bar{z}}{(1+|z|^{2})^{N+1}}&=2\pi \int_{r=0}^{\infty}\frac{2 r^{2j-1}}{(1+|r|^{2})^{N+1}}dr
    \end{align*}
    With: $$\phi(r)=\sum_{i=0}^{j-1}\frac{p_{i}r^{2i}}{(1+r^{2})^{N}}$$ and \begin{align*}
        p_{j-1}=-\frac{1}{2(N-j-1)},~~~~ p_{i}=\frac{(i+1)p_{i+1}}{N-i} ~~\text{ and }~~p_{0}=-\frac{1}{2 N \binom{N-1}{N-j}}
    \end{align*}
    then, $\phi'(r)=\frac{ r^{2j-1}}{(1+|r|^{2})^{N+1}}.$ Therefore: $$\int_{\mathbb{S}^{2}(M)}f_{j}(z,1)\bar{f}_{\ell}(z,1)\frac{id z\wedge d\bar{z}}{(1+|z|^{2})^{N+1}}=\frac{2\pi}{N\binom{N-1}{N-j}}\delta_{j,\ell}.$$
\end{proof}
Consequently, we obtain: \begin{align*}
    \mathcal{Z}_{\rho_{Ar},\rho_{can},N}(0)=\frac{e^{\frac{N(N+1)}{2}}}{2^{N}}\prod_{\ell=1}^{N}\frac{2\pi}{N\binom{N-1}{\ell-1}}=\frac{(2\pi)^{N} e^{N(N+1)/2}}{2^{N}N^{N}}\prod_{l=1}^{N}\frac{1}{\binom{N-1}{\ell-1}}.
\end{align*}

Hence $$\ln \mathcal{Z}_{\rho_{Ar},\rho_{can},N}(0)=\frac{N(N+1)}{2}+N\left(\ln2\pi-\ln2 \right)-N\ln(N!)+2G(N+1)$$ with $G$ the Barnes $G$-function. Moreover, \begin{equation}\label{Eq:Barnes}G(N+1)=\frac{N^{2}}{2}\ln N-\frac{3}{4}N^{2}+\frac{N}{2}\ln2\pi-\frac{1}{12}\ln N+\zeta'(-1)+O\left(\frac{1}{N} \right).\end{equation}

We can therefore compute: 
\begin{equation}
    \ln \mathcal{Z}_{\rho_{Ar},\rho_{can,N}}(0)=-\frac{1}{2}N\ln N+\left( \frac{3}{2}\ln2\pi +\frac{1}{2}-\ln2\right)N-\frac{1}{6}\ln N+2\zeta'(-1)-\frac{1}{12}+O\left(\frac{1}{N}\right),
\end{equation}
which is exactly the result we obtained previously.

\subsubsection{Explicit computation of $\det\frac{1}{2}\square_{L}$ for $g=0$}
In this Section we present a computation of the determinant $\det\frac{1}{2}\square_{L}$ for the sphere with the usual round metric. 
In the case of genus $g=0$, with the round metric (volume=$4\pi$), the eigenvalues of $\frac{1}{2}\square_{L}$ are \cite{WengDeterminantLaplacian,Ikeda1978SpectraAE}: $$\ell(\ell+k+1)$$ with $\ell\in \mathbb{N}$ and the multiplicity is $2\ell+k+1.$ Therefore: $$\zeta_{k}(s)=\sum_{\ell>0}\frac{2\ell+k+1}{\ell^{s}(\ell+k+1)^{s}}=\sum_{\ell>0}\left(\frac{1}{\ell^{s-1}(\ell+k+1)^{s}} +\frac{1}{\ell^{s}(\ell+k+1)^{s-1}}\right).$$ We will follow the computations of \cite[Appendix C]{WEISBERGER1987171}, the result should therefore be $$\zeta'_{k}(0)=4\zeta'(-1)-\frac{1}{2}(k+1)^{2}+\sum_{j=1}^{k+1}(2j-k-1)\ln(j).$$ We can compute, for $\mathfrak{Re}(s)>1$: 
\begin{align*}
    \zeta_{k}(s)&=\sum_{\ell>0}\frac{1}{\Gamma(s-1)\Gamma(s)}\int_{0}^{\infty}t^{s-2}e^{-\ell t}dt\int_{0}^{\infty}r^{s-1}e^{-(\ell+k+1)r}+\sum_{\ell>0}\frac{1}{\Gamma(s)\Gamma(s-1)}\int_{0}^{\infty}t^{s-1}e^{-\ell t}dt\int_{0}^{\infty}r^{s-2}e^{-(\ell+k+1)r}\\&=\frac{1}{\Gamma(s)\Gamma(s-1)}\int_{0}^{\infty}\int_{0}^{\infty}dtdr (tr)^{s-2}(t+r)\sum_{\ell>0}\exp(-\ell(t+r)-(k+1)r).
\end{align*}
Now, we denote: $t=\theta \lambda $ and $r=(1-\theta)\lambda.$ Then, $t+r=\lambda$ and: 
\begin{align*}
\zeta_{k}(s)&=\frac{1}{\Gamma(s)\Gamma(s-1)}\int_{0}^{1}d\theta\int_{0}^{\infty}\lambda d\lambda  \left(\theta (1-\theta)\lambda^{2}\right)^{s-2}\lambda \sum_{l>0} \exp\big(-l\lambda-(k+1)(1-\theta)\lambda \big)\\&= \frac{1}{\Gamma(s)\Gamma(s-1)}\int_{0}^{1}d\theta (\theta(1-\theta))^{s-2}\int_{0}^{\infty}d\lambda \lambda^{2s-2}\exp(-(k+1)(1-\theta)\lambda)\sum_{\ell>0}\exp(-\ell\lambda) \\&=\frac{1}{\Gamma(s)\Gamma(s-1)}\int_{0}^{1}d\theta \big(\theta(1-\theta) \big)^{s-2}\int_{0}^{\infty}d\lambda \lambda^{2s-2}\frac{\exp(-(k+1)(1-\theta)\lambda)}{1-e^{-\lambda}} \\&=\frac{1}{\Gamma(s)\Gamma(s-1)}\int_{0}^{\infty}d\lambda\frac{\lambda^{2s-2} e^{-(k+1)\lambda}}{1-e^{-\lambda}}\int_{0}^{1}d\theta \big(\theta(1-\theta) \big)^{s-2} e^{(k+1)\theta\lambda}\\&=\frac{1}{\Gamma(s)\Gamma(s-1)}\int_{0}^{\infty}d\lambda \frac{\lambda^{2s-2}e^{-(k+1)\lambda}}{1-e^{-\lambda}}\sum_{n=0}^{\infty}(k+1)^{n}\frac{\lambda^{n}}{n!}\int_{0}^{1}d\theta\theta^{n+s-2}(1-\theta)^{s-2} \\&=\frac{1}{\Gamma(s)}\sum_{n=0}^{\infty}\frac{(k+1)^{n}}{n!}\int_{0}^{\infty}d\lambda \frac{\lambda^{2s+n-2}e^{(-k+1)\lambda}}{1-e^{-\lambda}}\frac{\Gamma(n+s-1)}{\Gamma(n+2s-2)}.
\end{align*}

For $\mathfrak{Re}(s)>1$, the Hurwitz Zeta function \cite[Chapter 12] {ApostolAnalyticNumberLivre} is defined by: $$\zeta(s,a)=\sum_{n=0}^{\infty}\frac{1}{(n+a)^{s}}=\frac{1}{\Gamma(s)}\int_{0}^{1}t^{s-1}\frac{e^{-at}}{1-e^{-t}}dt.$$ Hence: 
\begin{align*}
    \zeta_{k}(s)=\frac{1}{\Gamma(s)}\sum_{n=0}^{\infty}\frac{(2s+n-2)\Gamma(s+n-1)}{n!}(k+1)^{n}\zeta(2s+n-1,k+2).
\end{align*}

Hence, we find: 
\begin{align*}
    \zeta_{k}(0)&=-\frac{k}{2}-\frac{2}{3},\end{align*} 
    which is consistent with Proposition \ref{ConstantLapMag}, and \begin{align*}
    \zeta'_{k}(0)&=4\zeta'(-1,k+2)-2(k+1)\zeta'(0,k+2)+2(k+1)\zeta(0,k+2)+\frac{(k+1)^{2}}{2}+\sum_{n=3}^{\infty}\frac{n-2}{n(n-1)}(k+1)^{n}\zeta(n-1,k+2).
\end{align*}
Now, from \cite[eq. (3.1) and (3.2)]{ChoiCertainFamilySeries2002}, we have: $$(k+1)\sum_{n=2}^{\infty}\frac{\zeta(n,k+2)}{n}(k+1)^{n}=-(k+1)\ln((k+1)!)+(k+1)^{2}\psi(k+2)$$ and $$2\sum_{n=2}^{\infty}\frac{\zeta(n,k+2)}{n(n+1)}(k+1)^{n+1}=2\zeta'(-1,k+2)-2\zeta'(-1)+2(k+1)\left(\frac{1}{2}-k-2-\ln((k+1)!)+\frac{1}{2}\ln2\pi \right)+(1+\psi(k+2))(k+1)^{2}.$$
Therefore: 
\begin{align*}
    \zeta'_{k}(0)&=4\zeta'(-1,k+2)-2(k+1)\zeta'(0,k+2)+2(k+1)\zeta(0,k+2)+\frac{(k+1)^{2}}{2}+(k+1)\sum_{n=2}^{\infty}\frac{\zeta(n,k+2)}{n}(k+1)^{n}\\&-2\sum_{n=2}^{\infty}\frac{\zeta(n,k+2)}{n(n+1)}(k+1)^{n+1}\\&=4\zeta'(-1,k+2)-2(k+1)\zeta'(0,k+2)+2(k+1)\zeta(0,k+2)+\frac{(k+1)^{2}}{2}-(k+1)\ln((k+1)!)+(k+1)\psi(k+2)\\&-\left(2\zeta'(-1,k+2)-2\zeta'(-1)+2(k+1)\left(\frac{1}{2}-k-2-\ln((k+1)!)+\frac{1}{2}\ln2\pi \right)+(1+\psi(k+2))(k+1)^{2} \right)\\&=2\zeta'(-1,k+2)+2\zeta'(-1)-2(k+1)\zeta'(0,k+2)+2(k+1)\zeta(0,k+2)-\frac{(k+1)^{2}}{2}-(k+1)\ln2\pi+3(k+1)\\&+(k+1)\ln((k+1)!)+2k(k+1).
\end{align*}
Moreover, we have: $$\zeta'(-1,k+2)=\zeta'(-1)+\sum_{n=1}^{k+1}n\ln n ~~ \text{ and } ~~\zeta'(0,k+2)=-\frac{1}{2}\ln2\pi+\ln((k+1)!) ~~ \text{ and } ~~ \zeta(0,k+2)=-\frac{1}{2}-(k+1).$$
Thus, 
\begin{align*}
    \zeta'_{k}(0)&=4\zeta'(-1)+2\sum_{n=1}^{k+1}n\ln n-(k+1)\ln((k+1)!)-\frac{(k+1)^{2}}{2}.
\end{align*}
Which is the result that we wanted.

We can also note that the sum could be computed like in \cite[Appendix C]{WEISBERGER1987171} by observing that: $$\sum_{n=3}^{\infty}\frac{n-2}{n(n-1)}(k+1)^{n}\zeta(n-1,k+2)=\int_{0}^{\infty}f_{k}(t)dt$$ 
where $$f_{k}(t)=\frac{e^{-(k+2)t}}{t^{2}(1-e^{-t})}\left((k+1)t+e^{(k+1)t}((k+1)t-2)+2 \right).$$
We can see that $$\lim_{t \to 0} f_{k}(t)=\frac{(k+1)^{3}}{6} ~~ \text{ and } ~~ \lim_{t\to \infty} f_{k}(t)=0.$$

\subsection{Genus $g=1$}
For the case of genus $g=1$, we also know the values of the determinant of the scalar and magnetic Laplacians. For the Arakelov metric on the torus, according to Equation \eqref{VarDetLapScaCste}, the determinant of the scalar Laplacian is given by \cite{OSGOOD1988148}:
\begin{equation}\label{DetLapScaArTore}
    \det \Delta_{\rho_{Ar}}=4\pi^{2}\mathfrak{Im}(\tau)^{2}|\eta(\tau)|^{8}
\end{equation} and we also know that  $\vol(\mathbb{T}^{2}_{\tau},\rho_{Ar})=4\pi^{2}\mathfrak{Im}(\tau)|\eta(\tau)|^{4}$ from Proposition \ref{Prop:VolumeArakelov}.
Hence: $$\left(\frac{\det\Delta_{\rho_{Ar}}}{\vol(\mathbb{T}^{2}_{\tau},\rho_{Ar})\mathfrak{Im}(\tau)}\right)^{1/2}=|\eta(\tau)|^{2}.$$ 
We know that for the flat metric $\rho_{0}$ of volume $2\pi$, the determinant of the magnetic Laplacian is equal to \cite{BerthomieuAnalyticTorsion}: $$\ln\det\square_{L,\rho_{0},h}=-\frac{k}{2}\ln k+\frac{k}{2}\ln2\pi-\frac{k}{2}\ln2.$$
Therefore, using Proposition \ref{ConstantLapMag} we find $$\ln \det\square_{L,\rho_{Ar},h}=-\frac{k}{2}\ln k+\frac{k}{2}\ln\big(
2\pi^{2}\mathfrak{Im}(\tau) |\eta(\tau)|^{4} \big).$$
Finally, we have that $\ln B_{1,k}=-k\ln2\pi.$ Hence from the bosonization formula, Proposition \ref{Pr:FormuleBoso}, we find: $$\ln \mathcal{Z}_{\rho_{Ar},\rho_{can},N}^{(\theta)}(0)=\ln\det\square_{L,\rho_{Ar},h}-\ln B_{1,k}+\frac{1}{2}\ln \left(\frac{\det\Delta_{\rho_{Ar}}}{\vol(\mathbb{T}^{2}_{\tau},\rho_{Ar})\mathfrak{Im}(\tau)} \right)$$ and therefore
\begin{equation}
    \ln \mathcal{Z}_{\rho_{Ar},\rho_{can},N}^{(\theta)}(0)=-\frac{k}{2}\ln k+\frac{k}{2}\ln\big(
8\pi^{4}\mathfrak{Im}(\tau) |\eta(\tau)|^{4} \big)+\ln(|\eta(\tau)|^{2}).
\end{equation}
We can note that this is an exact formula.

Moreover, we can also explicitly compute the Faltings' delta invariant (see Equation \eqref{FaltingDeltaEq}) for the torus using Equation \eqref{DetLapScaArTore}, we find: \begin{equation}\label{Eq:FaltingsTore}\delta(\mathbb{T^{2}_{\tau}})=-6\ln \big( \mathfrak{Im}(\tau)|\eta(\tau)|^{4} \big)-8\ln2\pi\end{equation} which coincides with \cite[p. 252]{Jorgenson1990AsymptoticBehaviorFaltingDelta} and \cite[p. 417]{FaltingsCalculusArithmeticSurface1984}.
\subsubsection{Direct computation for $g=1$}\label{AnnexeDirectTore}
Similarly to what we did for the sphere, we will compute directly the modified partition function for the torus without using the bosonization formula. To perform this computation, we will use results on theta functions that can be found in \cite{Fay1992KernelFA, FarkasTheta}.
For $g=1$, we have: 
$$\exp G_{\rho_{can}}(z,w)=\left| \frac{\theta_1(z - w, \tau)}{\eta(\tau)} \right| \exp \left( \frac{\pi}{4 \mathfrak{Im}(\tau)} (z - w - \bar{z} + \bar{w})^2 \right).$$ and $$\mu_{\rho_{Ar}}=2i\pi^{2}|\eta(\tau)|^{4}dz\wedge d\bar{z}.$$ Hence: $$\mathcal{Z}_{\rho_{Ar},\rho_{can},N}^{(\theta)}(0)=\frac{1}{N!}\int_{\mathbb{T}^{2}_{\tau}}\prod_{j=1}^{N}2i\pi^{2}|\eta(\tau)|^{4}dz_{j}\wedge d\bar{z}_{j}\left(\prod_{1\leq i <j \leq N}\left| \frac{\theta_{1}(z_{i} - z_{j}, \tau)}{\eta(\tau)} \right| \exp \left( \frac{\pi}{4 \mathfrak{Im}(\tau)} (z_{i} - z_{j} - \bar{z}_{i} + \bar{z}_{j})^2 \right)\right)^{2}\|\theta([L]-\sum z_{i}-\mathcal{D})\|^{2}.$$

Using the results of \cite[p. 116]{Fay1992KernelFA}, we can rewrite this as: 
\begin{align*}\mathcal{Z}^{(\theta)}_{\rho_{Ar},\rho_{can},N}(0)&=\frac{\left(2\pi^{2}|\eta(\tau)|^{2}\right)^{N}}{N!}|\eta(\tau)|^{2}\int_{\mathbb{T}^{2}_{\tau}}\prod_{j=1}^{N}idz\wedge d\bar{z}\left|\det\left(\theta\begin{bmatrix}
        \frac{j}{N}\\0
    \end{bmatrix}(Nz_{i},N\tau) \right) \right|^{2}\exp\left(-\frac{2N\pi}{\mathfrak{Im}(\tau)}\sum\mathfrak{Im}(z_{i})^{2} \right)  \\&=\left(2\pi^{2}|\eta(\tau)|^{2}\right)^{N}|\eta(\tau)|^{2}\det\left(\int_{\mathbb{T}^{2}_{\tau}}idz\wedge d\bar{z} \exp \left(-\frac{2N\pi}{\mathfrak{Im}(\tau)}\mathfrak{Im}(z)^{2} \right)\theta\begin{bmatrix}
        \frac{j}{N}\\0
    \end{bmatrix}(Nz,N\tau)\theta\begin{bmatrix}
        \frac{k}{N}\\0
    \end{bmatrix}(Nz,N\tau)\right) \\&=\left(2\pi^{2}|\eta(\tau)|^{2}\right)^{N}|\eta(\tau)|^{2}\left( \sqrt{\frac{2\mathfrak{Im}(\tau)}{N}}\right)^{N}.
\end{align*}
Hence: $$\ln\mathcal{Z}_{\rho_{Ar},\rho_{can},N}^{(\theta)}(0)=-\frac{N}{2}\ln N+\ln\left(2\pi^{2}\sqrt{2\mathfrak{Im}(\tau)}|\eta(\tau)|^{2} \right)N+\ln\left(|\eta(\tau)|^{2}\right).$$

Hence, we obtained the wanted result with a direct method, without computing the determinant of the magnetic Laplacian.
\subsection{Genus $g>1$}
For the case of genus $g>1$, there is no explicit computation of the magnetic Laplacian available for general line bundles. However, there are explicit computations of the determinant of the magnetic Laplacian for the case where $L=K^{n}$, with $K$ the canonical line bundle of degree $\deg(K)=2(g-1)$ and $n \in \frac{1}{2}\mathbb{N}.$ We also consider $\rho_{hyp}$ to be a hyperbolic metric on $M$ with $R_{\rho_{hyp}}=-2.$ Then the line bundle $L=K^{n}$ will be endowed with the Hermitian metric $h_{hyp}=(\rho_{hyp})^{-n}.$

In this case, the expression of the magnetic Laplacian has been computed in \cite{BolteDetLap} (see also \cite{OshimaNotesDetLap,FinskiQuillen2019}) and is given by: 
$$\ln\square_{K^{n},\rho_{hyp},h_{hyp}}=Z(n)e^{2(g-1)\tilde{c}_{n-1}}$$ for $n \in \frac{1}{2}\mathbb{N}$, hence $$\ln \det \square_{L,\rho_{hyp},h_{hyp}}=\ln Z(n)+2(g-1)\tilde{c}_{n-1}$$ with, for $n \in \mathbb{N}$, $$\tilde{c}_{n}=\frac{1}{2}\sum_{j=0}^{n-1}(2n-2j-1)\ln(2nj+2n-j^{2}-j)-\left(n+\frac{1}{2} \right)^{2}+\left(n+\frac{1}{2} \right)\ln(2\pi)+2\zeta'(-1)-2\sum_{j=1}^{n-1} \ln (j!)-\ln (n!)$$ and $$\ln Z(n)=o(1)$$ and if $n= [n]+\frac{1}{2}$ 
\begin{align*}
    \tilde{c}_{n}&=\sum_{j=0}^{[n]-1}([n]-j)\ln \left(2[n]j+2[n]-j^{2}+1 \right)-([n]+1)^{2}+([n]+1)\ln2\pi+2\zeta'(-1)-2\sum_{j=0}^{[n]}\ln(j!).
\end{align*}
Using Proposition \ref{VarDetFi} we can therefore see that with $h$ an admissible metric:  
$$\ln \det \square_{L,\rho_{Ar},h}=2(g-1)\tilde{c}_{n-1}+\mathcal{F}(\rho_{hyp},h_{hyp})-\mathcal{F}(\rho_{Ar},h)+o(1).$$
We will use this formula for $n \in \mathbb{N}$ to check our computation directly.
We can compute: 
$$\begin{cases}
    B(\rho_{hyp},h_{hyp})=\frac{k}{2(g-1)}\\[1.8mm]\mathcal{F}(\rho_{hyp},h_{hyp})=\frac{1}{2}k\ln k-\frac{\ln( 4\pi (g-1))}{2}k+\frac{2}{3}(1-g)\ln k-\frac{2}{3}(1-g)\ln(4\pi(g-1)) \\[1.8mm]B(\rho_{Ar},h)=2\pi k e^{-2\sigma_{Ar}} \\[1.8mm] \mathcal{F}(\rho_{Ar},h)=\frac{1}{2}k\ln k-k\int_{M}\mu_{\rho_{can}}\sigma_{Ar}+\frac{2}{3}(1-g)\ln k-\frac{4}{3}(1-g)\int_{M}\mu_{\rho_{can}}\sigma_{Ar}+\frac{1}{12\pi}\int_{M}\mu_{\rho_{Ar}}\sigma_{Ar}\Delta_{Ar}\sigma_{Ar}.\end{cases}$$
    Hence $$\mathcal{F}(\rho_{hyp},h_{hyp})-\mathcal{F}(\rho_{Ar},h)=k\left(\int_{M}\mu_{\rho_{can}}\sigma_{Ar}-\frac{\ln(4\pi(g-1)) }{2} \right)+\frac{4}{3}(1-g)\left( \int_{M}\mu_{\rho_{can}}\sigma_{Ar}-\frac{\ln(4\pi(g-1)) }{2} \right)-\frac{1}{12\pi}\int_{M}\mu_{Ar}\sigma_{Ar}\Delta_{Ar}\sigma_{Ar}$$
and (see \ref{AnnexeCalculCn})
$$2(g-1)\tilde{c}_{n-1}=-\frac{k}{2}\ln k +\frac{k}{2}\ln(2\pi(g-1))+\frac{2(g-1)}{3}\ln k -\frac{2(g-1)}{3}\ln(2\pi(g-1))+2(g-1)\left(\frac{7}{24}+\frac{1}{12}\ln2\pi+\zeta'(-1)\right)+o(1).$$
Consequently, for $n \in \mathbb{N}$:
\begin{align*}
    \ln\det\square_{L,\rho,h}&=-\frac{k}{2}\ln k+k\left(\int_{M}\mu_{\rho_{can}}\sigma_{Ar}-\frac{\ln2}{2}\right)+\frac{2(g-1)}{3}\ln k +\frac{4}{3}(1-g)\int_{M}\mu_{\rho_{can}}\sigma_{Ar}\\&+2(g-1)\left(\frac{7}{24}+\frac{1}{12}\ln2\pi+\frac{1}{3}\ln2+\zeta'(-1) \right)-\frac{1}{12\pi}\int_{M}\mu_{Ar}\sigma_{Ar}\Delta_{Ar}\sigma_{Ar}+o(1).\end{align*}

This appears to be consistent with what we found (see Theorem \ref{DevAsympLapMagAdm} and Equations \eqref{Eq:IntSigmaArDelta-PsiP}-\eqref{Eq:IntSigmaAr-PsiP-SAY}).

\subsubsection{Computation of the asymptotics of $\tilde{c}_{n}$}\label{AnnexeCalculCn}
We will compute the asymptotic expansion of $\tilde{c}_{n}$ appearing in the magnetic Laplacian with $g>1$ for $n \in \mathbb{N}$ that we used in the preceding. We limit ourselves to the case $n \in \mathbb{N}$ for simplicity and brevity but similar computations give the same result when $n \in \mathbb{N}+\frac{1}{2}$. We have $k=2n(g-1)$ and: 
$$\tilde{c}_{n}=\frac{1}{2}\sum_{j=0}^{n-1}(2n-2j-1)\ln(2nj+2n-j^{2}-j)-\left(n+\frac{1}{2} \right)^{2}+\left(n+\frac{1}{2} \right)\ln(2\pi)+2\zeta'(-1)-2\sum_{j=1}^{n-1} \ln (j!)-\ln (n!)-\left(\frac{n}{2}+\frac{1}{3}\right)\ln2.$$
and from Barnes G function properties:     $$2\sum_{j=1}^{n-1}\ln(j!)=n^{2}\ln n -\frac{3}{2}n^{2}+n\ln2\pi-\frac{1}{6}\ln n +2\zeta'(-1)+o(1) $$ and from Stirling's formula: $$\ln(n!)=n\ln n-n+\frac{1}{2}\ln n +\frac{1}{2}\ln(2\pi)+\frac{1}{12n}+o\left(\frac{1}{n} \right)$$
and $$\ln((2n)!)=2n\ln n+n(2\ln2-2)+\frac{1}{2}\ln n+\frac{1}{2}\ln 4\pi +\frac{1}{24n}+o\left(\frac{1}{n} \right).$$
Moreover:
\begin{equation}\label{Eq:jlnj}
\sum_{j=1}^{n}j \ln j=\frac{1}{12}-\zeta'(-1)+\left(\frac{n^{2}}{2}+\frac{n}{2}+\frac{1}{12} \right)\ln n-\frac{n^{2}}{4}+o(1)
\end{equation}and $$\sum_{j=1}^{2n}j\ln j=\frac{1}{12}-\zeta'(-1)+\left(2n^{2}+n+\frac{1}{12} \right)\ln(n)+\left(2n^{2}+n+\frac{1}{12} \right)\ln(2)-n^{2}+o(1).$$ We can rewrite: 
\begin{align*}
    \tilde{c}_{n}&=2n\ln (n!)-\left(n+\frac{1}{2}\right)\ln((2n)!)-2\sum_{j=1}^{n}j \ln(j)+\sum_{j=1}^{2n}j\ln(j)-\left(n+\frac{1}{2} \right)^{2}+\left(n+\frac{1}{2} \right)\ln(2\pi)+2\zeta'(-1) \\&-2\sum_{j=1}^{n-1}\ln (j!)
\end{align*}
Using the preceding identities, this gives: 
\begin{align*}
    \tilde{c}_{n}=-\frac{1}{2}n\ln n+\frac{\ln\pi}{2}n-\frac{1}{6}\ln n-\frac{5}{24}+\zeta'(-1)+\frac{1}{4}\ln \pi+\frac{1}{12}\ln2 +o(1).
\end{align*}
Hence: $$\tilde{c}_{n-1}=-\frac{1}{2}n\ln n+\left(\frac{\ln \pi}{2}\right)n+\frac{2}{3}\ln n+\frac{7}{24}-\frac{1}{4}\ln\pi+\frac{1}{12}\ln2+\zeta'(-1)+o(1)$$ and since $k=2n(g-1)$: $$2(g-1)\tilde{c}_{n-1}=-\frac{k}{2}\ln k +\frac{k}{2}\ln(2\pi(g-1))+\frac{2(g-1)}{3}\ln k -\frac{2(g-1)}{3}\ln(2\pi(g-1))+2(g-1)\left(\frac{7}{24}+\frac{1}{12}\ln2\pi+\zeta'(-1)\right)+o(1).$$

\section{Relations between Arakelov metric, canonical metric and hyperbolic metric for $g>1$}\label{AnnexeArakelovHyperbolic}
We will recall some relations using quantities of hyperbolic geometry between the Arakelov metric, the canonical metric and hyperbolic metrics derived by Jorgenson and Kramer. This could allow to us to reformulate our main results with quantities of hyperbolic geometry in the case $g>1$. All this Section is based on results of Jorgenson and Kramer \cite{Jorgenson2006NoncompletenessOT, JorgensonBoundsFaltings}.

We will consider $M$ a Riemann surface of genus $g>1$ and we denote: $$\rho_{Ar}=e^{2\sigma_{Ar}}\rho_{can},~~ \rho_{can}=e^{2\sigma_{CanHyp}}\rho_{hyp}, ~~ \text{ and } ~~ \rho_{Ar}=e^{2\sigma_{ArHyp}}\rho_{hyp}$$
with $$R_{\rho_{hyp}}=-2.$$

Then, we know the following result relating the Arakelov metric $\rho_{Ar}$ and the hyperbolic metric $\rho_{hyp}$.

\begin{Proposition}[{\cite[eq. (3)]{Jorgenson2006NoncompletenessOT}}]
    Let $\rho_{hyp}$ be the hyperbolic metric and $\rho_{Ar}=e^{2\sigma_{ArHyp}} \rho_{hyp}$. Then:
    $$2\sigma_{ArHyp}(x)=\frac{1}{g} \left(-4\pi(1-g)F(x)-\frac{\pi}{g}\int_{M}\mu_{\rho_{hyp}}(x)F(x)\Delta_{\rho_{hyp}}F(x) +\frac{c_{sel}-1}{1-g}-g\ln4\right)$$
    where
    $$c_{sel}=\lim_{s \to 1} \left(\frac{Z'(s)}{Z(s)}-\frac{1}{s-1} \right)$$
    with $Z$ being the Selberg zeta function, and
    $$F(x)=\int_{0}^{\infty}\left(HK_{hyp}(t;x) -\frac{1}{\vol(M,\rho_{hyp})}\right)dt$$
    where
    $$HK_{hyp}(t;x)=\sum_{\gamma \in \Gamma}K_{\mathbb{H}}(t;x,\gamma x)-K_{\mathbb{H}}(t;z,z)$$
    and
    $$K_{\mathbb{H}}(t;x,y)=\frac{\sqrt{2}e^{-t/4}}{(4\pi t)^{3/2}}\int_{d_{\mathbb{H}(x,y)}}^{\infty}\frac{re^{-r^{2}/4t}}{\sqrt{\cosh(r)-\cosh(d_{\mathbb{H}}(x,y))}}dr.$$
\end{Proposition}

Additionally:

\begin{Proposition} [{\cite[eq. (29)]{JorgensonBoundsFaltings}}]
    With $\rho_{can}=e^{2\sigma_{CanHyp}}\rho_{hyp}$, we have:
    $$2\sigma_{CanHyp}=\ln\left(\frac{1}{\vol(M,\rho_{hyp})}+\frac{1}{2g}\int_{0}^{\infty}K_{hyp}(t;x)dt\right)$$
    where
    $$K_{hyp}(t;x)=\sum_{\gamma \in \Gamma} K_{\mathbb{H}}(t;x,\gamma x)$$
    which can be rewritten as:
    $$2\sigma_{canHyp}=\ln \left(\frac{1}{4\pi(g-1)}+\frac{1}{2g}\Delta_{\rho_{hyp}}F(x) \right).$$
\end{Proposition}
    From the previous two propositions, we have:

\begin{Proposition}
The Arakelov metric is such that $\rho_{Ar}=e^{2\sigma_{Ar}}\rho_{can}$ with 
    \begin{align*}
    \sigma_{Ar}&=\sigma_{ArHyp}-\sigma_{canHyp}\\
    &=\frac{1}{2g} \left(-4\pi(1-g)F(x)-\frac{\pi}{g}\int_{M}\mu_{\rho_{hyp}}F\Delta_{\rho_{hyp}}F +\frac{c_{sel}-1}{1-g}-g\ln4\right) -\frac{1}{2}\ln\left(\frac{1}{4\pi(g-1)}+\frac{1}{2g}\Delta_{\rho_{hyp}}F(x)\right).
    \end{align*}
\end{Proposition}

    We can observe that we have have, since $R_{\rho_{hyp}}$ is constant:
    $$\Psi(x,\rho_{hyp})=0 ~~ \text{ and } ~~ \Psi_{P}(\rho_{hyp})=0.$$

Moreover:
\begin{Proposition}
    We have:
    $$\phi_{canHyp}(x)=-\frac{1}{g}F(x) ~~ \text{ and } ~~ \int_{M}\mu_{\rho_{hyp}} F=\frac{c_{sel}-1}{4\pi(g-1)}.$$
\end{Proposition}

From the previous propositions we can express the functionals:
$$S_{L}(\sigma_{canHyp},\rho_{hyp})=\frac{1}{2}\int_{M}\mu_{\rho_{hyp}}\ln\left(\frac{1}{4\pi(1-g)}+\frac{1}{2g}\Delta_{\rho_{hyp}}F(x) \right) \left(\frac{1}{2}\Delta_{\rho_{hyp}}\ln\left(\frac{1}{4\pi(1-g)}+\frac{1}{2g}\Delta_{\rho_{hyp}}F(x) \right)-2 \right)$$

and

$$S_{M}(\sigma_{canHyp},\phi_{canHyp},\rho_{hyp})=\int_{M}\mu_{\rho_{hyp}} \left(-2\pi\frac{(1-g)}{g^{2}}F \Delta_{\rho_{hyp}}F+2\left(\frac{1}{4\pi(1-g)}+\frac{1}{2g}\Delta_{\rho_{hyp}}F \right) \ln\left(\frac{1}{4\pi(1-g)}+\frac{1}{2g}\Delta_{\rho_{hyp}}F \right) \right).$$

Therefore,
since, from Propostion \ref{RicVar}: $$\Psi_{P}(\rho_{can})=\Psi_{P}(\rho_{hyp})+S_{L}(\sigma_{CanHyp},\rho_{hyp})-2\pi(1-g)S_{M}(\sigma_{canHyp},\phi_{canHyp},\rho_{hyp})$$
we obtain:
\begin{align*}
    \Psi_{P}(\rho_{can})&=\frac{1}{2}\int_{M}\mu_{\rho_{hyp}}\ln\left(\frac{1}{4\pi(1-g)}+\frac{1}{2g}\Delta_{\rho_{hyp}}F \right) \left(\frac{1}{2}\Delta_{\rho_{hyp}}\ln\left(\frac{1}{4\pi(1-g)}+\frac{1}{2g}\Delta_{\rho_{hyp}}F \right)-2 \right) \\
    &\quad -2\pi(1-g)\int_{M}\mu_{\rho_{hyp}} \left(-2\pi\frac{(1-g)}{g^{2}}F \Delta_{\rho_{hyp}}F+2\left(\frac{1}{4\pi(1-g)}+\frac{1}{2g}\Delta_{\rho_{hyp}}F \right) \ln\left(\frac{1}{4\pi(1-g)}+\frac{1}{2g}\Delta_{\rho_{hyp}}F \right) \right).
\end{align*}

We also have:
\begin{align*}
    \vol(M,\rho_{Ar})&=\exp\Bigg(-\frac{\pi}{g^{2}}\int_{M}\mu_{hyp} F \Delta_{hyp} F-\frac{c_{sel}-1}{g(g-1)}-2\ln2\Bigg)\int_{M}\mu_{hyp} \exp \bigg(-\frac{4\pi(g-1)}{g}F \bigg).
\end{align*}

Consequently, the preceding formulae can be put in the result obtained in Theorem \ref{ThéorèmeFinalSansTheta} for the case where $g\geq2.$

\printbibliography

Université de Strasbourg, Institut de recherche mathématique avancée, IRMA, Strasbourg, France.

\emph{e-mail}: lbourgoin@unistra.fr
\end{document}